\documentclass[11pt]{amsart}

\usepackage[colorlinks=true,
pdfstartview=FitV, linkcolor=cyan, citecolor=purple,
urlcolor=blue]{hyperref}
\usepackage{stmaryrd}
\usepackage{enumerate}
\usepackage[pdftex]{graphicx}
\usepackage{amsmath}
\usepackage[font=small,labelfont=bf, skip=0pt]{caption}
\usepackage{amssymb}
\usepackage{mathrsfs}
\usepackage{faktor}
\usepackage{amsthm}
\usepackage{relsize}
\usepackage{bbm}
\usepackage{color}
\usepackage{braket}
\usepackage[scale=0.8, top=2cm, bottom=2.5cm, left=2.5cm, right=2.5cm]{geometry}
\usepackage[pagewise]{lineno}
\usepackage[dvipsnames]{xcolor}
\usepackage{scalerel}

\newtheorem{thm}{Theorem}[section]

\newtheorem*{thm*}{Theorem}
\newtheorem{prop}[thm]{Proposition}
\newtheorem*{prop*}{Proposition}
\newtheorem{defn}[thm]{Definition}
\newtheorem{lem}[thm]{Lemma}

\newtheorem{que}[thm]{Question}
\newtheorem*{que*}{Question}
\newtheorem{example}[thm]{Example}

\newtheorem{conj}[thm]{Conjecture} 
\newtheorem{rem}[thm]{Remark}

\def\bb #1{\mathbb{#1}}
\def\bf #1{\mathbf{#1}}
\def\cal #1{\mathcal{#1}}

\def\frak #1{\mathfrak{#1}}
\def\rm #1{\mathrm{#1}}
\def\sf #1{\mathsf{#1}}
\def\wt #1{\widetilde{#1}}

\newcommand{\eee}{\varepsilon}

\numberwithin{equation}{section}

\title[Correlation for convex projective surfaces]{Correlation of the renormalized Hilbert length \\for convex projective surfaces}
 
\author[Dai]{Xian Dai}
\address{Mathematical institute of Heidelberg University, Heidelberg, Germany, 69117}
\email{xdai@mathi.uni-heidelberg.de}

\author[Martone]{Giuseppe Martone}
\address{Department of Mathematics, University of Michigan, Ann Arbor, MI 41809}
\email{martone@umich.edu}

\begin{document}
\maketitle

\begin{abstract} In this paper we focus on dynamical properties of (real) convex projective surfaces. Our main theorem provides an asymptotic formula for the number of free homotopy classes with roughly the same renormalized Hilbert length for two distinct convex real projective structures. The \emph{correlation number} in this asymptotic formula is characterized in terms of their Manhattan curve. We show that the correlation number is not uniformly bounded away from zero on the space of pairs of hyperbolic surfaces, answering a question of Schwartz and Sharp. In contrast, we provide examples of diverging sequences, defined via \emph{cubic rays}, along which the correlation number stays larger than a uniform strictly positive constant. In the last section, we extend the correlation theorem to Hitchin representations.
\end{abstract} 

\setcounter{tocdepth}{1}
\tableofcontents
\section{Introduction} \label{section introduction}
\thispagestyle{empty}

In this paper, we study the \emph{correlation} of length spectra of pairs of convex projective surfaces. We describe how length spectra of two convex projective structures $\rho_1$ and $\rho_2$ correlate on a closed connected orientable surface  $S$ of genus $g\geq 2$. More precisely, we investigate the asymptotic behavior of the number of closed curves whose $\rho_1$-Hilbert length and $\rho_2$-Hilbert length are roughly the same. This question was first considered by Schwartz and Sharp  in the context of hyperbolic surfaces in \cite{CorrelationHyperbolic} (see also \cite{Glor_altern}), and more generally for negatively curved metrics in \cite{DalBo, Poll_Sharp_Corr}.
The holonomy of a real convex projective structure is an example of {\em Hitchin representation} and
our correlation theorem holds in this more general setting.

We now discuss our results in greater detail.
A \emph{(marked real) convex projective structure} on the surface $S$ is described by a strictly convex set $\Omega_\rho$ in the real projective plane which admits a cocompact action by a discrete subgroup of $\rm{SL}(3,\bb R)$ isomorphic to $\Gamma=\pi_1(S)$. We denote by $\rho\colon\Gamma\to\rm{SL}(3,\bb R)$ the corresponding representation and by $X_\rho$ the surface $S$ equipped with the convex projective structure. Goldman \cite{Gold_convex} proved that the space of convex projective surfaces $\frak C(S)$ is an open cell of dimension $-8\chi(S)$. Hyperbolic structures on $S$ define, via the Klein model of the hyperbolic plane, a $-3\chi(S)$-dimensional subspace of $\frak C(S)$, called the {\em Fuchsian locus}, which is naturally identified with the Teichm{\"u}ller space $\cal T(S)$ of $S$. As customary, we will blur the distinction between $\mathcal T(S)$ and the Fuchsian locus.

Every convex projective structure $\rho$ induces a \emph{Hilbert length} $\ell^H_\rho$ for non-trivial conjugacy classes of group elements in $\Gamma$. Given $[\gamma]\in [\Gamma]$ a non-trivial conjugacy class, we define $\ell^H_\rho([\gamma])$ as the length of the unique closed geodesic with respect to the Hilbert metric on $X_\rho$ that corresponds to the free homotopy class of $[\gamma]$ on $S$. The \emph{(marked) Hilbert length spectrum} of $\rho$ is the function $\ell^H_\rho\colon [\Gamma] \to \bb R_{> 0}$. We investigate a slightly different notion of length.
Denoting the \emph{topological entropy} of (the Hilbert geodesic flow of) $\rho$ as
\[
h(\rho)=\limsup_{T\to\infty} \frac{1}{T}\log\#\{[\gamma]\in[\Gamma]\mid \ell_{\rho}^H([\gamma])\leq T\},
\]
we define the \emph{renormalized Hilbert length spectrum} of $\rho$ as $L^H_\rho=h(\rho)\ell^H_\rho$. 

Our first theorem concerns the correlation of the renormalized Hilbert length spectra of two different convex projective structures. 
\begin{thm}[Correlation Theorem for Convex Real Projective Structures]\label{thm:main} Fix a precision $\eee>0$. Consider two convex projective structures $\rho_1$ and $\rho_2$ on a surface $S$ with distinct renormalized Hilbert length spectra $L_{\rho_1}^{H}\neq L_{\rho_2}^{H}$. There exist constants $C=C(\eee, \rho_1,\rho_2)>0$ and $M=M(\rho_1,\rho_2) \in (0,1)$ such that 
\[
\# \Big\{[\gamma]\in[\Gamma]\, \Big\vert\, L_{\rho_1}^{H}([\gamma]) \in \big(x,x+h(\rho_1)\eee\big),\  L_{\rho_2}^{H}([\gamma]) \in \big(x, x+h(\rho_2)\eee\big)\Big\} \sim C \frac{e^{Mx}}{x^{3/2}}.
\]
where $f(x)\sim g(x)$ means $f(x)/g(x) \to 1$ as $x \to \infty$.
\end{thm}

\begin{rem}
\begin{enumerate}
\item 
There is an involution on the space of convex projective structures, called the \emph{contragredient involution}, which sends $\rho$ to $\rho^*=(\rho^{-1})^t$. Cooper-Delp and Kim \cite{CD_rigidity,Kim_rigidity} show that $L_{\rho_1}^H=L_{\rho_2}^H$ if and only if $\rho_2$ is $\rho_1$ or $\rho_1^{*}$.
\item
It follows from the generalized prime geodesic theorem \cite{Margulis_Thesis, Pollicot_Zeta}, that for any fixed $\eee$, if $\rho_1$ and $\rho_2$ converge to convex projective structures with the same renormalized Hilbert length spectrum, then $M(\rho_1,\rho_2)$ converges to one and $C(\eee,\rho_1,\rho_2)$ diverges.
\end{enumerate}
\end{rem}

The topological entropy of any hyperbolic structure equals to one \cite{Huber_entropy} and the renormalized Hilbert length coincides with the Hilbert length. The entropy varies continuously on $\frak C(S)$ and Crampon \cite{Crampon_entropy} shows that it is strictly less than one away from the Fuchsian locus. Nie and Zhang \cite{Nie_entropy,Zhang_entropy} prove that the entropy can be arbitrarily close to zero. In Example \ref{ex:badcorrelation}, we show that there exist a Fuchsian representation $\rho_1$ and a representation $\rho_2\in\frak C(S)$ with topological entropy different from one such that
\[
\lim_{x\to\infty}\#\Big\{[\gamma]\in[\Gamma]\, \Big\vert\, \ell_{\rho_1}^{H}([\gamma]) \in \big(x,x+\eee\big),\  \ell_{\rho_2}^{H}([\gamma]) \in \big(x, x+\eee\big)\Big\}=0.
\]
In particular, for these representations the size of the set $\Big\{[\gamma]\in[\Gamma]\, \Big\vert\, \ell_{\rho_i}^{H}([\gamma])\in(x,x+\epsilon),\ i=1,2\Big\}$ does not grow exponentially. Renormalized length spectra are natural objects from a dynamical point of view (see Remark \ref{LengthComparisonPhenomenon} for a more detailed justification) and thus they play a key role in our discussion.

We refer to the exponent $M(\rho_1,\rho_2)$ from Theorem \ref{thm:main} as the \emph{correlation number} of $\rho_1$ and $\rho_2$. An important goal of this paper is to study the correlation number as we vary $\rho_1$ and $\rho_2$ in $\frak C(S)$. One interesting question asked in \cite{CorrelationHyperbolic} for the case of the Teichm{\"u}ller space is whether the correlation number $M(\rho_1,\rho_2)$ is uniformly bounded away from zero as its arguments range over all hyperbolic structures. We answer this question in the negative in Section \ref{sec:corrnum}. We prove the following.
\begin{thm}[Decay of Correlation Number]\label{thm:pinching}
There exist sequences $(\rho_n)_{n=1}^\infty$ and $(\eta_n)_{n=1}^\infty$ in the Teichm{\"u}ller space $\cal T(S)$ such that the correlation number satisfies
\[
\lim_{n\to\infty} M(\rho_n,\eta_n)=0.
\]
\end{thm}

The sequences in Theorem \ref{thm:pinching} are given by pinching a hyperbolic structure along two different pants decompositions which are filling. Intuitively, these are two families of hyperbolic structures diverging from each other in the Teichm{\"u}ller space thus suggesting a small correlation number when going to infinity. A key to prove Theorem \ref{thm:pinching} is a characterization of Sharp \cite{Sharp_Manhattan} of the correlation number in terms of the Manhattan curve \cite{Burger_Manhattan}. In Theorem \ref{thm:Manhattan}, we extend Sharp's characterization of the correlation number to the case of two convex projective structures.

In Section \ref{sec:applications}, we study explicit examples of correlation number in the space of convex projective structures. In contrast to Theorem \ref{thm:pinching}, we provide pairs of diverging sequences for which the correlation numbers are uniformly bounded below away from zero. These sequences $(\rho_t)_{t\geq 0}$, called {\em cubic rays}, are defined using holomorphic cubic differentials via Labourie-Loftin's parameterization of $\frak C(S)$ \cite{Labourie_FlatProjective,Loftin_AffineSpheres}.  More precisely, Labourie and Loftin describe a mapping class group equivariant homeomorphism between $\frak C(S)$ and the vector bundle of holomorphic cubic differentials over $\cal T(S)$. The sequences $(\rho_t)_{t \geq 0}$ lie in fibers of $\frak C(S)$ with base point $\rho_0\in \cal T(S)$ and correspond to rays $(tq)_{t\geq 0}$ for $q$ a fixed $\rho_0$-holomorphic cubic differential. We will recall Labourie-Loftin's parameterization of $\frak C(S)$ and the definition of cubic rays more precisely in section \ref{sec:applications}.

Using work of Tholozan \cite{Tho_entropy}, we show in Lemma \ref{lem:tholo} that, for a cubic ray $(\rho_t)_{t\geq 0}$, the renormalized Hilbert length of $\rho_t$ is bi-Lipschitz to the one of $\rho_0$ with Lipschitz constants independent of $t$. We deduce that the correlation number of any two convex projective structures in different fibers is uniformly bounded from below by the correlation number of its base hyperbolic structures.

\begin{thm}\label{thm:differentfibers} Let $(\rho_t)_{t\geq 0}$, $(\eta_r)_{r\geq 0}$ be two cubic rays associated to two different hyperbolic structures $\rho_0\neq\eta_0$. Then, there exists a constant $C>0$ such that for all $t,r \geq 0$, 
\[
M(\rho_t,\eta_r)\geq CM(\rho_0,\eta_0).
\]
\end{thm}

A similar statement holds for most pairs of cubic rays with the same base hyperbolic structure. 

\begin{thm}\label{thm:cubicrays}  Let $\rho_t$ and $\eta_t$ be two cubic rays associated to two different holomorphic cubic differential $q_1$ and $q_2$ on a hyperbolic structure $X_0$ such that $q_1,q_2$ have unit $L^2$-norm with respect to $X_0$ and $q_1\neq -q_2$. Then there exists a constant $C>0$ such that for all $t>0$,
\[
M(\rho_{t},\eta_{t})\geq C.
\]
\end{thm}
 
In Section \ref{ssec:systole}, which is for the most part independent of the rest of the paper, we show how Lemma \ref{lem:tholo} can be used to study renormalized Hilbert geodesic currents $\nu_{\rho_t}$ along a cubic ray $(\rho_t)_{t>0}$. Geodesic currents are geometric measures on the space of complete geodesics on the universal cover of $S$ and each geodesic current $\nu$ has a corresponding length spectrum $\ell_\nu\colon[\Gamma]\to \bb R_{\geq 0}$ with {\em systole} $\rm{Sys}(\nu)=\inf_{c\in[\Gamma]}\ell_\nu(c)$. See \S\ref{ssec:currents} for a detailed discussion.
It follows from \cite{Bon_currents, BCLS_currents, MZ_currents} that for every convex projective structure $\rho$, there exists a unique geodesic current $\upsilon_\rho$ whose length spectrum coincides with the renormalized Hilbert length spectrum of $\rho$. We call $\upsilon_{\rho}$ the {\em renormalized Hilbert geodesic current} (also known as the \emph{renormalized Liouville current}) of $\rho$. We prove the following.

\begin{thm}\label{prop:degen} As $t$ goes to infinity, the renormalized Hilbert geodesic current $(\upsilon_{\rho_t})_{t\in>0}$ along a cubic ray $(\rho_t)_{t\geq 0}$ converges, up to subsequences, to a geodesic current $\upsilon$ with $\rm{Sys}(\upsilon)>0$.
\end{thm}
This should be compared with what happens for hyperbolic structures: given a sequence of hyperbolic structures which leaves every compact subset of $\mathcal T(S)$, up to rescaling and passing to subsequences, the corresponding sequence of geodesic currents converges to a geodesic current with vanishing systole. Burger, Iozzi, Parreau, and Pozzetti \cite{BIPP_compact} show that this fact no longer holds for general sequences of Hilbert geodesic currents. Theorem \ref{prop:degen} provides new examples of diverging sequences of convex projective structures whose associated geodesic currents converge (projectively and up to subsequences) to a geodesic current with positive systole. In light of \cite[Theorem 1.13]{ALS_orbifolds}, Theorem \ref{prop:degen} extends a theorem of Burger, Iozzi, Parreau, Pozzetti \cite[Theorem 1.12]{BIPP_compact} who prove that the Hilbert geodesic currents of a diverging sequence of convex projective structures of a triangle group converge (projectively and up to subsequences) to a geodesic current with positive systole.

Finally, in Section \ref{sec:Hitchin}, we generalize the correlation theorem (Theorem \ref{thm:main}) to Hitchin components for a large class of length functions. We will replace the Hilbert geodesic flow of a convex projective structure, which is Anosov, with more general metric Anosov translation flows. 

 Given a hyperbolic structure $\rho\in\mathcal T(S)$, seen as a representation $\rho\colon\Gamma\to\rm{PSL}(2,\bb R)$, we obtain a representation $i\circ\rho\colon \Gamma\to\rm{PSL}(d,\bb R)$ by post-composing $\rho$ with the unique (up to conjugation) irreducible representation $i\colon\rm{PSL}(2,\bb R)\to \rm{PSL}(d,\bb R)$. The Teichm\"uller space $\cal T(S)$ embeds in this way in the character variety of $S$ and $\rm{PSL}(d,\bb R)$. The connected component $\mathcal H_d(S)$ of the character variety containing this image is known as the {\em Hitchin component}. Hitchin \cite{Hit_topol} showed that $\mathcal H_d(S)$ is homeomorphic to an open cell of dimension $-(\dim \rm{PSL}(d,\bb R))\chi(S)$. Choi and Goldman \cite{CG_convex} identify $\mathcal H_3(S)$ with the space $\frak C(S)$ which will be our main focus from Section \ref{sec:prelim} to Section \ref{sec:applications}.

In order to state the general correlation theorem for Hitchin representations (Theorem \ref{thm:HitchinMain}), we need to introduce some Lie theoretical notations. 

Let 
\[
\frak a=\{\vec x\in\bb R^d\mid x_1+\dots+x_d=0\}\qquad\text{ and }\qquad\frak a^+=\{\vec x\in\frak a\mid x_1\geq \dots\geq x_d\}
\]
denote the (standard) Cartan subspace for $\rm{PSL}(d,\mathbb{R})$ and the (standard) positive Weyl chamber, respectively. Let $\lambda\colon\rm{PSL}(d,\mathbb{R})\to\frak a^+$ be the \emph{Jordan projection} given by $\lambda(g)=(\log \lambda_{1}(g),\dots, \log \lambda_{d}(g))$
consisting of the logarithms of the moduli of the eigenvalues of $g$ in nonincreasing order. We consider linear functionals in
\[
\Delta=\left\{c_1\alpha_1+\dots+c_{d-1}\alpha_{d-1} \mid c_i\geq 0, \sum_ic_i>0\right\},
\]
where $\alpha_i\colon\frak a\to \bb R$ are the \emph{simple roots} defined by $\alpha_i(\vec x)=x_i-x_{i+1}$ with $i=1,\dots, d-1$. Observe that if $\phi\in\Delta$, then $\phi(\vec x)>0$ for all $\vec x$ in the interior of $\frak a^+$. The \emph{length function} $\ell^\phi_\rho$ for $\phi\in\Delta$ and $\rho\in \cal H_d(S)$ is defined by $\ell^\phi_\rho([\gamma])=\phi(\lambda(\rho(\gamma)))$. The length function is strictly positive because $\lambda(\rho(\gamma))$ is in the interior of $\frak a^+$ for all $[\gamma]\in[\Gamma]$ (see \cite{FG, Lab_Anosov}). The \emph{topological entropy} $h^{\phi}(\rho)$ and \emph{renormalized $\phi$-length} $L^{\phi}_{\rho}([\gamma])=h^{\phi}(\rho)\ell^{\phi}_{\rho}([\gamma])$ are defined in similar manner as for convex real projective structures. Given a Hitchin representation $\rho$, we denote by $\rho^{*}$ its contragredient given by $\rho^{*}=(\rho^{-1})^{t}$.

We are now ready to state the correlation theorem for Hitchin representations.
\begin{thm}\label{thm:HitchinMain}
Given a linear functional $\phi\in\Delta$ and a fixed precision $\eee>0$, for any two different Hitchin representations $\rho_1, \rho_2:\Gamma \to \rm{PSL}(d,\mathbb{R})$  such that $\rho_2\neq \rho_1^*$, there exist constants $C=C(\eee, \rho_1,\rho_2,\phi)>0$ and $M=M(\rho_1,\rho_2, \phi) \in (0,1)$ such that 
\[
\# \Big\{[\gamma]\in[\Gamma]\, \Big\vert\, L_{\rho_1}^{\phi}([\gamma]) \in \big(x,x+h^{\phi}(\rho_1)\eee\big),\  L_{\rho_2}^{\phi}([\gamma]) \in \big(x, x+h^{\phi}(\rho_2)\eee\big)\Big\} \sim C \frac{e^{Mx}}{x^{3/2}}.
\]
\end{thm}

\begin{rem}
\begin{enumerate}
    \item For $\rho\in \cal H_d(S)$, when $\phi=\alpha_i$ and $i=1,\dots, d-1$, we know that $h^{\alpha_i}(\rho)=1$ thanks to \cite[Thm B]{PotrieSambarino}. Thus, $L^{\alpha_i}_\rho=\ell^{\alpha_i}_\rho$ is the simple root length without any renormalization. The correlation theorem in this case can be seen as a natural generalization of Schwartz and Sharp's correlation theorem for hyperbolic surfaces.
    \item Note that Theorem \ref{thm:main} is a corollary of Theorem \ref{thm:HitchinMain} when we set $d=3$ and consider the positive root $\phi(\vec x)=(\alpha_1+\alpha_2)(\vec x)=x_1-x_3$.
\end{enumerate}
\end{rem}

It would be interesting to extend the results on the correlation number proved in sections \ref{sec:corrnum} and \ref{sec:applications} to the context Hitchin representations. In this spirit, we end the paper by raising Question \ref{Que:SimpleRootLengths} and Conjecture \ref{Conj:cyclic} which are motivated by Theorem \ref{thm:pinching} and Theorem \ref{thm:cubicrays}, respectively.

\subsection*{Structure of the paper}

In Sections \ref{sec:prelim} through \ref{sec:applications}, we focus on convex projective structures. In this case, the length function can be defined geometrically via the Hilbert distance and Benoist proved in \cite{BenoistI} that the associated geodesic flow is Anosov. These two facts will simplify the exposition. The main results (Theorems \ref{thm:pinching}, \ref{thm:differentfibers} and \ref{thm:cubicrays}) in Sections \ref{sec:corrnum} and \ref{sec:applications} concern the behavior of the correlation number along geometrically defined sequences of convex projective structures. We establish the correlation theorem in full generality in Section \ref{sec:Hitchin}, after recalling the precise definitions of Hitchin representations and their length functions and the theory of metric Anosov translation flows.

\section{Preliminaries  for convex real projective structures}\label{sec:prelim} 
Consider a connected, closed, oriented surface $S$ with genus $g\geq 2$ and denote by $\Gamma$ its fundamental group. In this preliminary section we focus on convex projective structures on $S$. We will discuss in section \ref{sec:Hitchin} how parts of the material presented here hold for general Hitchin components.

The structure of Section \ref{sec:prelim} is as follows. In \S\ref{ssec:convexprojective}, we briefly recall the relevant geometric aspects of the theory of convex projective structures on surfaces. We refer to \cite{BenoistI,Gold_convex} for further details and background. In \S\ref{ssec:reparam}, we collect dynamical properties of convex projective surfaces which will play an important role in sections \ref{sec:corrthm} and \ref{sec:corrnum}. We briefly survey the theory of {\em geodesic currents} in \S\ref{ssec:currents} which will be used in the proof of Theorem \ref{thm:pinching} and in \S\ref{ssec:systole}. Finally, in \S \ref{ssec:independence} we prove the independence lemma (Lemma \ref{thm:indep}) which plays a key role in the proof of Theorem \ref{thm:main}.

\subsection{Convex projective surfaces}\label{ssec:convexprojective}

A \emph{properly convex set} $\Omega$ in $\mathbb {RP}^2$ is a bounded open convex subset of an affine chart. A properly convex set whose boundary does not contain open line segments is \emph{strictly convex}. We will exclusively focus on strictly convex sets in this paper. We equip a strictly convex set $\Omega$ with its \emph{Hilbert metric $d_\Omega$}. More precisely, if $x,y\in\Omega$, the projective line $\overline{xy}$ passing through $x$ and $y$ intersects the boundary of $\Omega$ in two points $a,b$ where $a,x,y,b$ appear in this order along $\overline{xy}$. The Hilbert distance between $x$ and $y$ is
\[
d_\Omega(x,y)=\frac{1}{2}\log [a,x,y,b]
\]
where $[a,x,y,b]$ denotes the crossratio of four points on a projective line. With the Hilbert metric, geodesics are segments of a projective line intersecting $\Omega$. Typically, the Hilbert metric is not Riemannian, but it derives from a Finsler norm. Thus, one can study the \emph{unit tangent bundle} $T^1\Omega$ of a strictly convex set $\Omega$.

The main objects of interest are representations $\rho\colon \Gamma\to \rm{SL}(3,\bb R)$ such that $\rho(\Gamma)$ preserves a properly convex set $\Omega_\rho$ on which it acts properly discontinuously with quotient homeomorphic to the closed surface $S$. In this case, we say that $\rho$ is a \emph{(marked real) convex projective structure} which \emph{divides} $\Omega_\rho$ and denote by $X_\rho$ the surface $S$ equipped with the convex projective structure $\rho$. Since $S$ is a closed surface of negative Euler characteristic, $\Omega_\rho$ is strictly convex and if $\gamma\in\Gamma$ is non-trivial, then the moduli $\lambda_1(\rho(\gamma))>\lambda_2(\rho(\gamma))>\lambda_3(\rho(\gamma))>0$ of the eigenvalues of $\rho(\gamma)$ are distinct. (See for example \cite[3.2 Theorem]{Gold_convex} and references therein). 

The Hilbert distance on $\Omega_\rho$ induces the \emph{Hilbert length spectrum $\ell_\rho^H$} for non-trivial conjugacy classes of group elements in $\Gamma$. Algebraically, if $[\Gamma]$ is the set of conjugacy classes of non-identity elements in $\Gamma$, namely $[\gamma]\in [\Gamma]$ is the conjugacy class of $\gamma\neq\text{id}$, then
\[
\ell_\rho^H([\gamma])=\frac{1}{2}\log\frac{\lambda_1(\rho(\gamma))}{\lambda_3(\rho(\gamma))}.
\]

Benoist \cite{BenoistI} proved that if $\rho$ is a convex projective structure on $S$, then $(\Omega_\rho,d_{\Omega_\rho})$ is Gromov hyperbolic, the Gromov boundary and the topological boundary of $\Omega_\rho$ coincide, and $\partial\Omega_\rho$ is of class $C^{1+\alpha}$ for some $\alpha\in(0,1]$. For a(ny) point $o\in \Omega_\rho$, the orbit map $\tau_o$ is a quasi-isometric embedding. It follows that the induced \emph{limit map} between Gromov boundaries $\xi_\rho\colon\partial\Gamma\to\partial\Omega_\rho$ is a $\rho$-equivariant bi-H\"older homeomorphism. 

If $\rho(\Gamma)$ divides a strictly convex set $\Omega_\rho$ in $\bb{RP}^2$, then $\rho^\ast(\Gamma)=(\rho(\Gamma)^{-1})^t$ divides a (typically different) strictly convex set $\Omega_{\rho^\ast}$. We refer to this operation on the space of convex projective structures as the \emph{contragredient involution}. The contragredient involution preserves the Hilbert length because for all $[\gamma]\in[\Gamma]$
\[
\ell^H_\rho([\gamma])=\frac{1}{2}\log\frac{\lambda_1(\rho(\gamma))}{\lambda_3(\rho(\gamma))}=\frac{1}{2}\log \frac{\frac{1}{\lambda_3(\rho(\gamma))}}{\frac{1}{\lambda_1(\rho(\gamma))}}=\ell^H_{\rho^\ast}([\gamma]).
\]
A standard computation using the irreducible representation $\rm{PSL}(2,\bb R)\to \rm{PSL}(3,\bb R)$ shows that if $\rho$ is a hyperbolic structure, then $\rho=\rho^\ast$. The converse holds by \cite[Thm 1.3]{Benoist_cvxcones}: if $\rho$ is a convex projective structure such that $\rho=\rho^*$, then $\rho$ is a hyperbolic structure. 

\subsection{The Hilbert geodesic flow and reparametrization function}\label{ssec:reparam}

Suppose that $\rho$ is a convex projective surface. The \emph{Hilbert geodesic flow} $\Phi^{\rho}$ is defined on the unit tangent bundle of the surface $T^{1}X_{\rho}$. The image $\Phi^\rho_{t}(w)$ of a point $w=(x,v)$ is obtained by following the unit speed geodesic for time $t$ leaving $x$ in the direction $v$. When it is clear from context, we simply write $\Phi^{\rho}$ as $\Phi$.  

 The Hilbert geodesic flow $\Phi$ on $T^{1}X_{\rho}$ is an example of a topologically mixing Anosov flow by \cite[Prop. 3.3 and 5.6]{BenoistI}. A standard reference for the theory of Anosov flows is \cite[\S 6]{Katok_dynamics}.
A key property for our discussion is that topologically mixing Anosov flows can be modeled by Markov partitions and symbolic dynamics in the sense of Bowen \cite{Bowen_symbolicHyperFlow}.  

Given a positive H\"older continuous function $f\colon T^1X_\rho\to \bb R$, one can define a {\em H{\"o}lder reparametrization} of the flow $\Phi$ by time change. We construct the flow $\Phi^{f}$ following \cite[section 2]{Sambarino_Quantitative}. First, we define $\kappa:T^1 X_{\rho} \times\mathbb{R}\to \mathbb{R} $ as
\[
\kappa(x,t)=\int_{0}^{t} f(\Phi_{s}(x))\mathrm{d}s.
\]
Given the fact that $f$ is positive and $T^{1}X_{\rho}$ is compact, the function $\kappa(x,\cdot)$ is an increasing homeomorphism of $\mathbb{R}$. We therefore have an inverse $\alpha: T^1 X_{\rho} \times\mathbb{R}\to \mathbb{R} $ that verifies
\[
\alpha(x,\kappa(x,t))=\kappa(x,\alpha(x,t))=t
\]
for every $x\in T^1X_{\rho} \times \mathbb{R}$. The H{\"o}lder reparametrization of $\Phi$  by the H{\"o}lder continuous function $f$ is given by $\Phi^{f}_t(x)= \Phi_{\alpha(x,t)}(x)$. We say that $f$ is a {\em reparametrization function} for $\Phi^f$. The new flow $\Phi^{f}=\{\Phi_t^{f}\}_{t\in\mathbb{R}}$ shares the same set of periodic orbits of $\Phi$. For any periodic orbit $\tau$ of $\Phi$ with period $\lambda(\tau)$, its period as a $\Phi^{f}$ periodic orbit is
\begin{equation}\label{eqn:NewPeriod}
\lambda(f,\tau)=\int_{0}^{\lambda(\tau)} f(\Phi_s(x))\mathrm{d}s. 
\end{equation}
Property (\ref{eqn:NewPeriod}) is a simple application of the definitions of $\alpha(x,t)$ and $\kappa(x,t)$.
\begin{rem} Each oriented closed geodesic $\gamma$ on a convex projective structure $\rho$ is associated with a periodic orbit $\tau$ of $\Phi$. On the other hand, an oriented closed geodesic $\gamma$ corresponds to a free homotopy class $[\gamma]\in[\Gamma]$. We adopt different perspectives depending on necessity in this paper while keeping in mind that they are the same object described from different points of view.
\end{rem}

\begin{rem}\label{rem:ReserveRepa}
\begin{enumerate}
    \item One can check from the definition that $\Phi$ is a H{\"o}lder reparametrization of $\Phi^{f}$ with the H{\"o}lder reparametrization function given by $1/f$.
    \item Set $\Psi=\Phi^f$ and consider $g$ a positive H\"older reparametrization of $\Psi$, then $\Psi^g$ is a H\"older reparametrization of $\Phi$ by the H\"older function $g\cdot f$.
\end{enumerate}
\end{rem}
Let $\rho_1$ and $\rho_2$ be two convex projective structures on a surface $S$. The next lemma states that there exists a positive H\"older continuous reparametrization function $f_{\rho_1}^{\rho_2}: T^{1}X_{\rho_1} \to \mathbb{R}$ encoding the Hilbert length spectrum of $\rho_2$.

\begin{lem}\label{lem: rep} Let $\rho_1$ and $\rho_2$ be convex projective structures on a surface $S$. There exists a positive H{\"o}lder continuous function $f_{\rho_1}^{\rho_2}: T^{1}X_{\rho_1} \to \mathbb{R}$ such that for every periodic orbit $\tau$ corresponding to $[\gamma]\in [\Gamma]$ one has
\[
\lambda\left(f_{\rho_1}^{\rho_2},\tau\right)=\ell_{\rho_2}^{H}([\gamma]).
\]
\end{lem}
\begin{proof} This is a standard argument which we include for the sake of completeness.
Let us lift the picture to the universal cover. By \cite[Equation (20)]{BenoistI} and since the limit map of $\rho_1$ is bi-H\"older, there exists a H\"older continuous $\rho_1$-equivariant homeomorphism $\chi: T^1\Omega_{\rho_1}\to \partial^3 \Omega_{\rho_1}$ where $\partial^3\Omega_{\rho_1}$ is the set of ordered triples of distinct points in $\partial\Omega_{\rho_1}$. Since the Hilbert geodesic flow is Anosov, it follows from \cite[Theorem 3.2]{Sambarino_Quantitative} (see also \cite[Proposition 5.21]{Intro_pressure}) that for any choice of an auxiliary hyperbolic surface $\rho_0$, there exists a H\"older continuous positive reparametrization $g^{\rho_2}_{\rho_0}\colon T^1\Omega_{\rho_0}\cong\partial^3\bb H^2\to\bb R$ of the geodesic flow of $\rho_0$ with periods $\ell_{\rho_2}^{H}([\gamma])$. Considering the H\"older continuous function $\xi^{(3)}\colon\partial^3 \Omega_{\rho_1}\to \partial^3\bb H^2$ induced by the inverse of the limit map $\xi_{\rho_1}$ and the limit map $\xi_{\rho_0}\colon\partial\Gamma\to\partial\bb H^2$, we obtain the composition $g^{\rho_2}_{\rho_0}\circ\xi^{(3)}\circ \chi\colon T^1\Omega_{\rho_1}\to\bb R$ which is the lift of the desired reparametrization function $f^{\rho_2}_{\rho_1}$. The equality $\lambda\left(f_{\rho_1}^{\rho_2},\tau\right)=\ell_{\rho_2}^{H}([\gamma])$ follows from equivariance of the limit maps. 
\end{proof}

\subsection{Thermodynamic formalism}\label{ssec:thermo}
 In this subsection, we will introduce several important concepts from thermodynamic formalism in our context that will be needed later. Standard references for thermodynamic formalism and Markov codings are \cite{Bowen-LectureNotes, Pollicot_Zeta}.

For a continuous function $f: T^{1}X_{\rho} \to \mathbb{R}$, we define its \emph{pressure} with respect to $\Phi$ as
\begin{equation*}
 P(\Phi,f)=\limsup\limits_{T\longrightarrow \infty} \frac{1}{T}\log\Big(\sum\limits_{\tau\in R_T} e^{\lambda(f,\tau)}\Big)
\end{equation*}
where $R_T:=\{\tau \text{ } \text{periodic orbit of }\Phi\text{ }|\text{ }\lambda(\tau)\in[T-1, T]\}$. One can check that the topological entropy $h(\rho)$ is $P(\Phi,0)$. For simplicity, we omit the geodesic flow $\Phi$ from the notation and write $P(\cdot)$ for $P(\Phi, \cdot)$. The pressure can be characterized as follows.
\begin{prop}[Variational principle]
\label{prop defn_pressure}
The pressure of a continuous function $f:T^1X_{\rho}\to \mathbb{R}$ satisfies
\begin{equation*}
P(f)= \sup\limits_{\mu\in \mathcal{M}^{\Phi}}\Big(h(\mu)+ \int f \mathrm{d}\mu\Big)
\end{equation*}
where $\mathcal M^\Phi$ is the space of $\Phi$-invariant probability measures on $T^1X_\rho$ and $h(\mu)=h(\Phi,\mu)$ denotes the measure-theoretic entropy of $\Phi$ with respect to $\mu\in\mathcal{M}^{\Phi}$.
\end{prop}

 A $\Phi$-invariant probability measure $
\mu$ on $T^1X_{\rho}$ is called an \emph{equilibrium state} for $f$ if the supremum is attained at $\mu$.

\begin{rem}
\begin{enumerate}
    \item For a H{\"older} continuous function $f:T^1X_{\rho}\to \mathbb{R}$ there exists a unique equilibrium state $\mu_f$ by \cite[Theorem 3.3]{Bowen-Ruelle}.
 \item The equilibrium state $\mu_0$ for $f=0$ is called a {\em probability measure of maximal entropy} or {\em Bowen-Margulis measure}, denoted as $\mu_{\Phi}$. The Hilbert geodesic flow $\Phi$ is topologically mixing and Anosov, thus admits a unique measure of maximal entropy on $T^{1}X_{\rho}$. The entropy of the measure of maximal entropy coincides with the topological entropy. See for instance \cite[Section 20]{Katok_dynamics}.
 \end{enumerate}
\end{rem}

The following lemma, derived from Abramov's formula \cite{Abr_formula}, allows us to rescale a reparametrization function to be pressure zero. 

\begin{lem}[{Sambarino \cite[Lemma 2.4]{Sambarino_Quantitative}, Bowen-Ruelle \cite[Proposition 3.1]{Bowen-Ruelle}}]\label{lem:PressueZero} For a positive H\"older reparametrization function $f$ on $T^{1}X_{\rho}$ and $h\in\bb R$, the pressure function satisfies
\[P(-hf)=0\]
if and only if $h=h(\Phi^{f})$,
where 
\[
h(\Phi^{f})=\limsup_{T\to\infty} \frac{1}{T}\log\#\{\tau \text{ periodic orbit} \mid \lambda(f,\tau)\leq T\}.
\]
By definition, $h(\Phi^{f})$ is the topological entropy of the reparametrized flow $\Phi^{f}$.
\end{lem}
We will use Lemma \ref{lem:PressueZero} in the proofs of Theorems \ref{thm:main}, Theorem \ref{thm:pinching} and Theorem \ref{thm:cubicrays}.

\begin{rem} By construction, if $f=f^{\rho_2}_{\rho_1}$ is the reparametrization function defined in Lemma \ref{lem: rep}, then the topological entropy of the flow $\Phi^f$ is equal to the topological entropy $h(\rho_2)$ of $\rho_2$ as defined in the introduction.
\end{rem}
Finally, we introduce \emph{(Liv\v{s}ic) cohomology}. We say two H\"older continuous functions $f$ and $g$ are  \emph{(Liv\v{s}ic) cohomologous} if there exists a H{\"o}lder continuous function $V\colon T^1X_{\rho}\to \mathbb{R}$ that is differentiable in the flow's direction such that 
  \begin{equation*}
      f(x)-g(x)= \frac{\partial}{\partial t}\bigg|_{t=0}V(\Phi_{t}(x)).
  \end{equation*}
  
\begin{rem}\label{rem Liv}
\begin{enumerate}
    \item \emph{(Liv\v{s}ic's Theorem, \cite{Liv_ic_Cohomology})} Two H\"older continuous functions $f$ and $g$ are cohomologous on $T^1X_{\rho}$ if and only if $\lambda(f,\tau)=\lambda(g,\tau)$ for any periodic orbit $\tau$ of $\Phi$. It follows that the pressure of a H\"older continuous function depends only on its cohomology class.
    \item  Two H\"older continuous functions $f$ and $g$ have the same equilibrium state on $T^1X_{\rho}$ if and only if $f-g$ is cohomologous to a constant $C$. In this case, we have $P(f)=P(g)+C$ \cite[Section 20]{Katok_dynamics}.
\end{enumerate}
\end{rem}  

\subsection{Geodesic currents for convex projective surfaces}\label{ssec:currents}

Fix an auxiliary hyperbolic structure $m$ on $S$. A \emph{geodesic current} is a Borel, locally-finite, $\pi_1(S)$-invariant measure on the set of complete geodesics of the universal cover $\wt S$. An important example is the geodesic current $\delta_\gamma$ given by Dirac measures on the axes of the lifts of a closed geodesic $\gamma$ in $S$.

The space $\cal C(S)$ of geodesic currents is a convex cone in an infinite dimensional vector space. Bonahon \cite{Bonahon_annals} extended the intersection pairing on closed curves to the space of geodesic currents, i.e. there exists a positive, symmetric, bilinear pairing
\[
i\colon \cal C(S)\times\cal C(S)\to \bb R_{\geq 0}
\]
such that $i(\delta_c,\delta_d)$ equals the intersection number of the closed geodesics $c$ and $d$.

Extending work of Bonahon \cite{Bon_currents}, in \cite{BCLS_currents,MZ_currents} it was shown that for each convex projective surface $\rho$ there exists a \emph{Hilbert geodesic current} $\nu_\rho$ such that for every $[\gamma]\in[\Gamma]$  
\[
i(\nu_\rho,\delta_\gamma)=\ell_\rho^H([\gamma]),
\]
where  $\gamma$ denotes the unique closed geodesic in its free homotopy class $[\gamma]$. Bonahon \cite[Proposition 15]{Bon_currents} proves that the geodesic current $\nu_{\rho}$ of a hyperbolic structure $\rho$ has self-intersection $i(\nu_{\rho},\nu_{\rho})=-\pi^2\chi(S)$. On the other hand, if $\rho\in\frak C(S)$ is not in the Fuchsian locus, then $i(\nu_\rho,\nu_\rho)>-\pi^2\chi(S)$ by Corollary 5.3 in \cite{BCLS_currents}.

In general, given a geodesic current $\nu$, we can use the intersection number to define its length spectrum $\ell_\nu\colon[\Gamma]\to\bb R^+$ as $\ell_\nu([\gamma])=i(\nu,\delta_{\gamma})$. The \emph{systole} of $\nu$ is then $
\rm{Sys}(\nu):=\inf_{[\gamma]\in[\Gamma]}\ell_\nu([\gamma])$.
Corollary 1.5 in \cite{BIPP_compact} shows that $\rm{Sys}\colon\cal C(S)\to[0,\infty)$ is a continuous function. 

A geodesic current is {\em period minimizing} if for all $T>0$ the set $ \#\{[\gamma]\in[\Gamma]\mid \ell_{\nu}([\gamma])<T\}$ is finite.
We define the \emph{exponential growth rate} of a period minimizing geodesic current $\nu$ by
\[
h(\nu)=\lim_{T\to \infty}\frac{1}{T}\log \#\{[\gamma]\in[\Gamma]\mid \ell_{\nu}([\gamma])<T\}.
\]
The notation is motivated by the fact that if $\rho$ is a convex projective structure and $\nu_\rho$ is the corresponding Hilbert geodesic current, then $h(\nu_\rho)$ is equal to the topological entropy $h(\rho)$ of $\rho$. 

The systole and the exponential growth rate of a geodesic current are related by the following inequality, which will play an important role in the proof of Theorem \ref{thm:pinching}.
\begin{thm}[Corollary 7.6 in \cite{MZ_currents}]\label{thm:entropysystole} Let $S$ be a closed, connected, oriented surface of genus $g\geq 2$. There exists a constant $C>0$ depending only on $g$ such that for every period minimizing geodesic current $\nu\in\mathcal C(S)$
\[
\text{Sys}(\nu)h(\nu)\leq C.
\]
\end{thm}

\subsection{Independence of convex projective surfaces}\label{ssec:independence}
We start by recalling the notion of (topologically) weakly mixing flows which motivates the concept of independence of representations. A flow $\varphi$ on $T^{1}X_{\rho}$ is \emph{weakly mixing} if its periods do not generate a discrete subgroup of $\mathbb{R}$. In particular, we can ask whether the Hilbert geodesic flow $\Phi$ is weakly mixing. Equivalently, we ask whether there exists some non-zero real number $a\in\mathbb{R}$ such that $a\ell_{\rho}^{H}(\gamma)\in \mathbb{Z}$ for all $[\gamma]\in\Gamma$. In the proof of Theorem \ref{thm:main}, we will need a strengthening of this property for a pair of representations.

Two convex projective structures $\rho_1$ and $\rho_2$ are \emph{dependent} if there exist $a_1,a_2\in\mathbb R$, not both equal to zero, such that $a_1\ell^H_{\rho_1}([\gamma])+a_2\ell^H_{\rho_2}([\gamma])\in\mathbb Z$ for all $[\gamma]\in[\Gamma]$. Otherwise, $\rho_1$ and $\rho_2$ are \emph{independent} over $\mathbb Z$. The next definition clarifies that this notion of independence is of a dynamical nature.

\begin{defn} Two positive H\"older continuous functions $f_1,f_2\colon T^1X_\rho\to\bb R$ are \emph{dependent} if there exists $a_1,a_2\in \bb R$ not both equal to zero and a  complex valued $C^1$ function $u: T^1X_\rho\to S^1$ such that $a_1f_1+a_2f_2=\frac{1}{2\pi i}\frac{u'}{u}$.  
Here, $u'$ denotes the derivative of $u$ along the flow, i.e. $\frac{\partial}{\partial t}\big\vert_{t=0} u\circ \Phi_t$. Otherwise, $f_1$ and $f_2$ are \emph{independent}. In particular, $f$ is said to be \emph{(in)dependent} if $f$ and the constant function $g \equiv 1$ are (in)dependent.
\end{defn}

\begin{rem}
The integral over a closed orbit of $\frac{u'}{u}$ is an integer multiple of $2\pi i$. Thus, by integrating along closed orbits and using Equation (\ref{eqn:NewPeriod}) and Remark \ref{rem:ReserveRepa}, we see that if $\rho_1$ and $\rho_2$ are independent, then the reparametrization function $f_{\rho_1}^{\rho_2}$ on $T^1X_{\rho_1}$ is independent.
\end{rem}

We now prove that convex projective surfaces with distinct Hilbert length spectra are independent.

\begin{lem}[Independence lemma]\label{thm:indep} Let $\rho_1$ and $\rho_2$ be convex projective structures with distinct Hilbert length spectra. If there exist $a_1,a_2\in\bb R$ such that $a_1\ell^H_{\rho_1}([\gamma])+a_2\ell^H_{\rho_2}([\gamma])\in\bb Z$ for all $[\gamma]\in[\Gamma]$, then $a_1=a_2=0$. 
\end{lem}
\begin{proof} Our proof follows from combining results of Benoist and an argument of Glorieux from \cite{Glo_lin_indep}. 

We prove this statement by contradiction. Consider the product representation $\eta=\rho_1\times\rho_2\colon \Gamma\to\rm G_1\times\rm G_2$ where $\rm G_i$ denotes the Zariski closure of $\rho_i$. Benoist \cite[Thm 1.3]{Benoist_cvxcones} proved that $\rm G_i$ is $\rm {PSL}(3,\bb R)$ or isomoprhic to $\rm{PSO}(1,2)$. Either way, this Zariski closure is connected and simple, so $\rm G_1\times\rm G_2$ is semi-simple. Choose a Cartan subspace of $\rm G_1\times\rm G_2$ such that
\[
\frak a\subseteq\{(\vec x,\vec y)\in\bb R^3\times\bb R^3\mid x_1+x_2+x_3=0=y_1+y_2+y_3\}
\]
and a positive Weyl chamber $\frak a^+$ contained in $\{(\vec x,\vec y)\in\frak a\mid x_1\geq x_2\geq x_3\text{ and }y_1\geq y_2\geq y_3\}$.
Denote by $\lambda\colon \rm G_1\times\rm G_2\to \frak a^+$ the corresponding Jordan projection
and by $\phi^H_{a_1,a_2}$ the non-zero linear functional $\phi^H_{a_1,a_2}(\vec x,\vec y)=a_1(x_1-x_3)+a_2(y_1-y_3)$ so that $\phi^H_{a_1,a_2}(\lambda(\eta(\gamma)))=a_1\ell^H_{\rho_1}([\gamma])+a_2\ell^H_{\rho_2}([\gamma])$.

Denote by $\rm H\subseteq \rm G_1\times\rm G_2$ the Zariski closure of $\eta(\Gamma)$.
As a first step, observe that $\rm H\neq \rm G_1\times\rm G_2$. Otherwise, the Proposition on page 2 of \cite{Benoist_propasy} implies directly that $\lambda(\eta(\Gamma))$ is dense (in the standard topology) in $\frak a$. We obtain a contradiction as we assumed $\phi^H_{a_1,a_2}(\lambda(\eta(\gamma)))\in\bb Z$ for all $\gamma\in\Gamma$ and $\phi^H_{a_1,a_2}$ is continuous.

Let $\pi_i\colon \rm H\to\rm G_i$ for $i=1,2$ denote the projection maps and note that $\pi_i(\eta(\Gamma))=\rho_i(\Gamma)$. In particular, $\pi_i$ is surjective since $\rm H$ and $\rm G_i$ are the Zariski closures of $\eta(\Gamma)$ and $\rho_i(\Gamma)$, respectively. Denote by $\rm N_1=\pi_2^{-1}(\text{id})$ (resp. $\rm N_2=\pi_1^{-1}(\text{id})$) the kernel of $\pi_2$ (resp. $\pi_1$) which is naturally identified with a normal subgroup of $\rm G_1$ (resp. $\rm G_2$) (note the indices in the definition of $\rm N_i$). Then, Goursat's lemma \cite[Thm 5.5.1]{Hall} states that the image of $\rm H$ in $\rm G_1/\rm N_1\times \rm G_2/\rm N_2$ is the graph of an isomorphism $\rm G_1/\rm N_1\cong \rm G_2/\rm N_2$. Since $\rm G_1$ is simple, then $\rm N_1=\{e\}$ or $\rm G_1$.

\emph{Case 1:} Suppose $\rm N_1=\rm G_1$. Then $\rm N_2=\rm G_2$ and $\rm H$ is the direct product $\rm G_1\times\rm G_2$, which is a contradiction.

\emph{Case 2:}  Suppose $\rm N_1=\{e\}$. Since $\rm G_2$ is simple, $\rm N_2=\{e\}$ and $\rm G_1\cong \rm G_2$. In other words, $\rm H$ is the graph of an automorphism $\iota\colon\rm G_1\to\rm G_2$. This induces an automorphism of the corresponding Lie algebras and, by the classification of their outer automorphisms given in \cite{LieClassification} (see also \cite[Theorem 11.9]{PressureMetric-MainPaper}), we deduce that $\rho_2$ is conjugated to either $\rho_1$ or $\rho_1^\ast$, which contradicts our hypothesis.
\end{proof}

Observe that Lemma \ref{thm:indep} readily implies that the geodesic flow of a convex projective structure is weakly mixing. Otherwise there exist $a\in\mathbb R$, $a\neq 0$ such that $a\ell^H_\rho([\gamma])\in\mathbb Z$ which directly contradicts Lemma \ref{thm:indep} with $a_1=a$, $a_2=0$, $\rho_1=\rho$ and $\rho_2\in\frak C(S)$ different from $\rho_1$ and $\rho_1^\ast$.

\section{The Correlation Theorem}\label{sec:corrthm}
In this section, we study the length spectra of two convex real projective structures simultaneously. 

This idea appeared first in \cite{CorrelationHyperbolic} for studying correlation of hyperbolic structures. We adapt their argument to the context of convex real projective structures. Theorem \ref{prop:lalley}, which was proved independently by Lalley and Sharp with slightly different conditions, gives the asymptotic formula for the number of closed orbits of an Axiom A flow under constraints. Anosov flows are important examples of Axiom A flows and Theorem \ref{prop:lalley} will be a crucial ingredient for our proof of Theorem \ref{thm:main}.

Fix a convex projective surface $\rho$. Let $f\colon T^{1}X_{\rho}\to\bb R$ be a H\"older continuous function and consider the function $t \to P(tf)$ for $t\in \mathbb{R}$, where $P$ denotes the pressure. This function is real analytic and strictly convex in $t$ when $f$ is not cohomologous to a constant \cite[Prop 4.12]{Pollicot_Zeta}. Its derivative satisfies
\begin{equation}\label{eq:deriv pressure}
P'(tf):=\frac{d}{dt}P(tf)=\int f d\mu_{tf},
\end{equation}
where $\mu_{tf}$ is the equilibrium state for $tf$. We denote by $J(f)$ the open interval of values $P'(tf)$. If $a \in J(f)$, we let $t_a \in\mathbb{R}$ be the unique real number for which $P'(t_a f)=\int f\mathrm{d}\mu_{t_af}=a$. We ease notation and set $\mu_a=\mu_{t_af}$.

The following is the key  result needed to establish our Theorem \ref{thm:main} (and Theorem \ref{thm:HitchinMain}).

\begin{thm}[{Lalley \cite[Theorem I]{Lalley}, Sharp \cite[Theorem 1]{PrimeOrbitThm}}]\label{prop:lalley}
Let $f: T^{1}X_{\rho} \to \mathbb{R}$ be an independent H\"older continuous function and let $a \in J(f)$. Then, for fixed $\eee >0$, there is a constant $C=C(f,\eee)$ such that $$ \#\{ \tau: \lambda(\tau) \in (x, x+\eee),\ \lambda(f, \tau) \in (ax, ax+\eee)\} \sim C \frac{\exp(h(\mu_a)x)}{x^{3/2}}.$$
\end{thm}
\begin{proof}
Recall that the Hilbert geodesic flow on $T^{1}X_{\rho}$ is Anosov and, as pointed out at the end of section \ref{sec:prelim}, weakly mixing. Thus, we can apply Lalley and Sharp's results which hold for all weakly mixing Axiom A flows.
\end{proof}
\begin{rem}
The constant $C=C(f,\eee)>0$ has the same order of magnitude as $\eee^2$ and is related to $P''(t_af)$. See \cite[Thm 5]{Lalley} and \cite{CorrelationHyperbolic}.
\end{rem}

We introduce the concepts of \emph{pressure intersection} and \emph{renormalized pressure intersection} which we will need for the proof of Theorem \ref{thm:main}. 
\begin{defn} Let $\rho_1$ and $\rho_2$ be two convex projective structures and let $f\colon T^1X_{\rho_1}\to \bb R$ be a H\"older continuous reparametrization function. The \emph{pressure intersection} of $\rho_1$ and $\rho_2$ is 
\[
\mathbf{I}(\rho_1,\rho_2):=\int f \mathrm{d}\mu_{\Phi^{\rho_1}}
\]
where $\mu_{\Phi^{\rho_1}}$ is the measure of maximal entropy for $\Phi^{\rho_1}$. 
The \emph{renormalized pressure intersection} of $\rho_1$ and $\rho_2$   is
\[
\mathbf{J}(\rho_1,\rho_2):=\frac{h(\rho_2)}{h(\rho_1)}\mathbf{I}(\rho_1,\rho_2).
\]
\end{defn}
By Liv\v{s}ic's Theorem, the definitions of pressure intersection and renormalized pressure intersection do not depend on the choice of reparametrization function (see Remark \ref{rem Liv}).
\begin{prop}[{\cite[Proposition 3.8]{PressureMetric-MainPaper}}]\label{thm:Jrigidity} For every $\rho_1,\rho_2\in\frak C(S)$, the renormalized pressure intersection is such that 
\[
\mathbf{J}(\rho_1,\rho_2)\geq 1
\]
with equality if and only if $L^H_{\rho_1}=L^H_{\rho_2}$.
\end{prop}
\begin{rem}\label{LengthComparisonPhenomenon}
Given two distinct elements $\rho_1, \rho_2$ in $\cal T(S)$, we can always find some free homotopy classes $[\gamma_1]$ and $[\gamma_2]$ so that $\ell_{\rho_1}([\gamma_1])>\ell_{\rho_2}([\gamma_1])$ and $\ell_{\rho_1}([\gamma_2])<\ell_{\rho_2}([\gamma_2])$. However, 
Tholozan \cite[Theorem B]{Tho_entropy} shows that there exist representations $\rho$ in $\frak C(S)$ which {\em dominate} a Fuchsian representation $j\in \cal T(S)$ in the sense that $\ell^{H}_{\rho}([\gamma]) \geq \ell_{j}([\gamma])$ for all $[\gamma]\in \pi_1(S)$ (see also Section \ref{sec:applications}). However, Proposition \ref{thm:Jrigidity} implies that whenever $\rho_1$ and $\rho_2$ in $\frak C(S)$ have distinct renormalized Hilbert length spectra we can always find some free homotopy classes $[\gamma_1]$ and $[\gamma_2]$ so that $L^H_{\rho_1}([\gamma_1])>L^H_{\rho_2}([\gamma_1])$ and $L^H_{\rho_1}([\gamma_2])<L^H_{\rho_2}([\gamma_2])$.
This motivates our focus on the renormalized Hilbert length as in the proof of the correlation theorem \ref{thm:main} below.
 \end{rem}

Now we are ready to prove our main theorem stated in the introduction and repeated below for the convenience of the reader.

\medskip\noindent
{\textbf{Theorem} \ref{thm:main}\textbf{.}} {\em Fix a precision $\eee>0$. Consider two convex projective structures $\rho_1$ and $\rho_2$ on a surface $S$ with distinct renormalized Hilbert length spectra $L_{\rho_1}^{H}\neq L_{\rho_2}^{H}$. There exist constants $C=C(\eee, \rho_1,\rho_2)>0$ and $M=M(\rho_1,\rho_2) \in (0,1)$ such that 
\[
\# \Big\{[\gamma]\in[\Gamma]\mid L_{\rho_1}^{H}([\gamma]) \in \big(x,x+h(\rho_1)\eee\big),\  L_{\rho_2}^{H}([\gamma]) \in \big(x, x+h(\rho_2)\eee\big)\Big\} \sim C \frac{e^{Mx}}{x^{3/2}}
\]
where $f(x)\sim g(x)$ means $f(x)/g(x) \to 1$ as $x \to \infty$.}

\begin{proof} This proof is inspired by the proof for hyperbolic surfaces in \cite{CorrelationHyperbolic}. Our first goal is to show that, for the reparametrization function $f_{\rho_1}^{\rho_2}$ described in Lemma \ref{lem: rep}, the value $a$ in Theorem \ref{prop:lalley} can be chosen to be $\frac{h(\rho_1)}{h(\rho_2)}$. In order to ease notations, we set $f=f_{\rho_1}^{\rho_2}$ and we write $\Phi$ for the Hilbert geodesic flow $\Phi^{\rho_1}$ on $T^{1}X_{\rho_1}$. 

By Lemma \ref{lem:PressueZero}, we have that
\[
0=P(-h(\rho_2)f)= h\big(\mu_{-h(\rho_2)f}\big)- \int h(\rho_2) f \mathrm{d} \mu_{-h(\rho_2) f} \]
where $\mu_{-h(\rho_2)f}$ is the equilibrium state of $-h(\rho_2)f$. Hence 
\[
h(\rho_2) \int f \mathrm{d}\mu_{-h(\rho_2) f}=h(\mu_{-h(\rho_2) f})
\]
and 
\begin{equation}\label{eqn:upper}
\int f \mathrm{d}\mu_{-h(\rho_2) f}= \frac{h(\mu_{-h(\rho_2)f})}{h(\rho_2)} \leq  \frac{h(\rho_1)}{h(\rho_2)}.
\end{equation}
Notice the equality can be attained only when $\mu_{-h(\rho_2)f}=\mu_{\Phi}$, where $\mu_{\Phi}$ is the measure of maximal entropy for the geodesic flow $\Phi$. This happens only if $-h(\rho_2)f$ is cohomologous to a constant which, by integrating, implies the length spectrum $\ell_{\rho_1}^{H}$ is a multiple of $\ell_{\rho_2}^{H}$. This is impossible by Lemma \ref{thm:indep} and therefore the inequality is strict.

On the other hand, by Propositon \ref{thm:Jrigidity}, we have that
\begin{align*}
1 \leq \mathbf{J}(\rho_1,\rho_2) &= \frac{h(\rho_2)}{h(\rho_1)} \mathbf{I}(\rho_1,\rho_2)=\frac{h(\rho_2)}{h(\rho_1)} \int f\mathrm{d}\mu_{\Phi}.
\end{align*}
where $\mu_{\Phi}$ is the Bowen-Margulis measure of $\Phi$. This yields
\begin{equation}\label{eqn:lower}
\int f \mathrm{d}\mu_{\Phi} \geq \frac{h(\rho_1)}{h(\rho_2)}. 
\end{equation}
Moreover, the equality is attained only when $h(\rho_2)f$ is cohomologous to $h(\rho_1)$ which is impossible given our hypotheses.

Combining the inequalities (\ref{eqn:upper}) and (\ref{eqn:lower}), together with the fact that $J(f)$ is an open interval, we conclude that $\left(\frac{h(\rho_1)}{h(\rho_2)}-\delta, \frac{h(\rho_1)}{h(\rho_2)} 
+ \delta\right) \subset J(f)$ for small $\delta>0$. In particular, because $P'(tf)$ is a strictly increasing continuous function, there exists some $t_{a_0}\in (-h(\rho_2),0)$ such that $a_0=P'(t_{a_0}f)=\frac{h(\rho_1)}{h(\rho_2)}$, as desired.

We now show that $0<h(\mu_{a_0}) <h(\rho_1)$. Because the Hilbert geodesic flow is Anosov, it has positive entropy with respect to any equilibrium state of a H{\"o}lder reparametrization, so $0<h(\mu_{a_0})\leq h(\mu_{\Phi})=h(\rho_1)$ and the equality occurs if and only if $\mu_{a}=\mu_{\Phi}$ is the measure of maximal entropy. But as shown above,
\[
\int f \mathrm{d}\mu_{\Phi} > \frac{h(\rho_1)}{h(\rho_2)}=\int f \mathrm{d}\mu_{a_0}.
\]
This shows that $\mu_{a_0}$ can not be the Bowen-Margulis measure and hence $h(\mu_{a_0}) <h(\rho_1)$.

Thanks to Lemma \ref{thm:indep} and Theorem \ref{prop:lalley}, we can conclude that for $\rho_1$, $\rho_2$ convex projective structures such that $\rho_2\neq\rho_1,\rho_1^\ast$, there exists $\wt C=\wt C(\rho_1,\rho_2,\eee)$ such that
\[
\#\left\{[\gamma]\in[\Gamma]\colon \ell_{\rho_1}^H([\gamma])\in (y,y+\eee),\ \ell_{\rho_2}^H([\gamma])\in \left(\frac{h(\rho_1)}{h(\rho_2)}y,\frac{h(\rho_1)}{h(\rho_2)}y+\eee\right)\right\}\sim \wt C \frac{\exp(h(\mu_{a_0})y)}{y^{3/2}}.
\]
Setting $x=h(\rho_1)y$ and clearing denominators, we have the desired statement with
$C=h(\rho_1)^{3/2}\wt C$ and $M=\frac{h(\mu_{a_{0}})}{h(\rho_1)}\in{(0,1)}$.
\end{proof}
\section{Correlation number, Manhattan curve, and Decay of Correlation number}\label{sec:corrnum}
In this section we focus on the correlation number $M(\rho_1,\rho_2)$ from Theorem \ref{thm:main}. In Theorem \ref{thm:Manhattan} we express the correlation number in terms of Burger's \emph{Mahnattan curve} \cite{Burger_Manhattan} generalizing the main result in \cite{Sharp_Manhattan}. In \S\ref{ssec:decay} we prove that the correlation number is \emph{not} uniformly bounded away from zero in $\frak C(S)$. Specifically, we provide two sequences of \emph{hyperbolic structures} along which the correlation number goes to zero, thus answering a question from \cite{CorrelationHyperbolic}.
\subsection{Manhattan curve and Correlation number}
Let $\rho_1$ and $\rho_2$ be two convex projective structures. The \emph{Manhattan curve of $\rho_1,\rho_2$} is the curve $\cal C(\rho_1,\rho_2)$ that bounds the convex set 
\[
\bigg\{(a,b) \in \bb R^2: \sum_{[\gamma]\in[\Gamma]} e^{-\big(a\ell^{H}_{\rho_1}([\gamma])+b\ell^{H}_{\rho_2}([\gamma])\big)}< +\infty\bigg\}
\]
Equivalently (see \cite{Sharp_Manhattan}), the Manhattan curve can be defined in terms of the pressure function and the reparametrization function $f=f_{\rho_1}^{\rho_2}: T^{1}X_{\rho_1} \to \mathbb{R}$ from Lemma \ref{lem: rep} as
\[
\cal C(\rho_1,\rho_2)=\bigg\{(a,b)\in\bb R^2\colon P(-a-bf)=0\bigg\}=\bigg\{(a,b)\in\bb R^2\colon P(-bf)=a\bigg\}.
\]

The next theorem collects the  properties of the Manhattan curve we will need which were first discussed in the setting of representations into isometry groups of rank one symmetric spaces (see \cite[Theorem 1]{Burger_Manhattan}).

\begin{thm}  Let $\rho_1,\rho_2\in\frak C(S)$ denote two convex projective surfaces.
\label{PropertyManhattan}
\begin{enumerate}
    \item The Manhattan curve is a real analytic convex curve passing through the points $(h(\rho_1),0)$ and $(0,h(\rho_2))$. 
    \item The normals to the Manhattan curve at the points $(h(\rho_1),0)$ and $(0, h(\rho_2))$ have slopes $\mathbf{I}(\rho_1, \rho_2)$ and $1/\mathbf{I}(\rho_2, \rho_1)$, respectively.
    \item The Manhattan curve is strictly convex if and only if $\rho_2\neq \rho_1, \rho_1^\ast$.
\end{enumerate}
\end{thm}
\begin{proof} Real analyticity and convexity of the Manhattan curve follow from its definition and real analyticity of the pressure function. The fact that $\mathcal C(\rho_1,\rho_2)$ passes through $(h(\rho_1),0)$ and $(0,h(\rho_2))$ follows from Lemma \ref{lem:PressueZero} and the definition of topological entropy.  To see that the normal to the Manhattan curve at $(h(\rho_1),0)$ has slope ${\bf I}(\rho_1,\rho_2)$, observe that the Manhattan curve is the graph $\cal C(\rho_1,\rho_2)$ of a real analytic function $q$ such that $P(-q(s)f)=s$. Then, differentiate the equality $P(-q(s)f)=s$ using equation (\ref{eq:deriv pressure}). The rest of item (2) follows by symmetry. Item (3) follows from item (2), Theorem \ref{thm:indep} and Proposition \ref{thm:Jrigidity}.
\end{proof}

Sharp \cite{Sharp_Manhattan} expressed the correlation number $M(\rho_1,\rho_2)$ in terms of the Manhattan curve for hyperbolic structures. We establish an analogous result for convex real projective structures.
\begin{thm}\label{thm:Manhattan} Let $\rho_1$ and $\rho_2$ be convex projective structures with distinct renormalized Hilbert length spectra.  Then, their correlation number can be written as
\[
M(\rho_1,\rho_2)=\frac{a}{h(\rho_1)}+\frac{b}{h(\rho_2)}
\]
where $(a,b)\in \cal C(\rho_1,\rho_2)$ is the point on the Manhattan curve at which the tangent line is parallel to the line passing through $(h(\rho_1),0)$ and $(0,h(\rho_2))$. See Figure \ref{fig:Manhattan}.
\end{thm}
\begin{figure}[!htb]
    \centering
    \includegraphics[width=.5\textwidth]{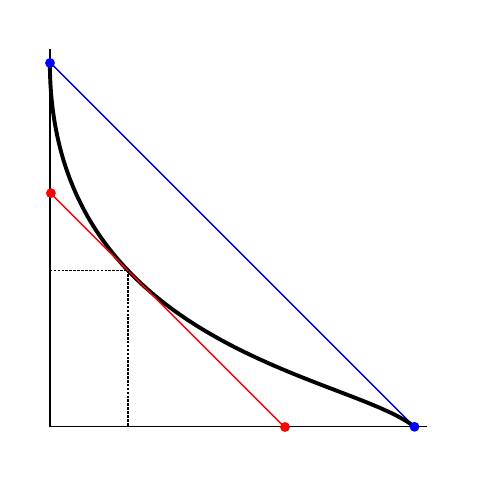}
    \put (-544,16){\makebox[0.7\textwidth][r]{\small{$0$}}}
    \put (-500,18){\makebox[0.7\textwidth][r]{\small{$a$}}}
    \put (-544,99){\makebox[0.7\textwidth][r]{\small{$b$}}}
    \put (-544,201){\makebox[0.7\textwidth][r]{\small{$h(\rho_2)$}}}
    \put (-350,16){\makebox[0.7\textwidth][r]{\small{$h(\rho_1)$}}}
    \caption{The Manhattan curve and the point $(a,b)$ described in Theorem \ref{thm:Manhattan}}
    \label{fig:Manhattan}
\end{figure}
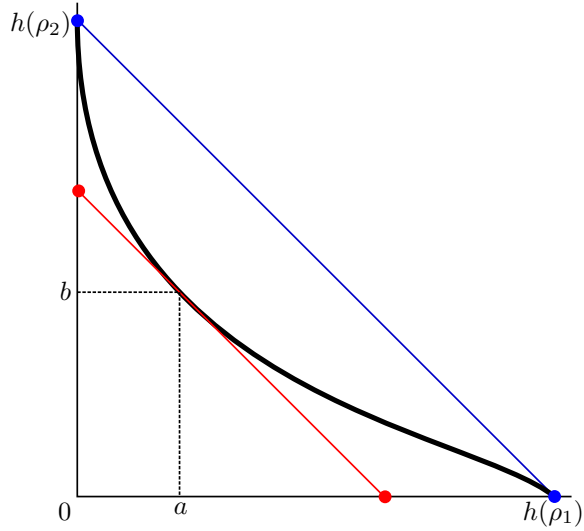    
\begin{proof} By the proof of Theorem \ref{thm:main}, the correlation number is such that 
$h(\rho_1)M(\rho_1,\rho_2)=h(\mu_{a_0})$, where $a_0=\int f \mathrm{d}\mu_{a_0}=\frac{h(\rho_1)}{h(\rho_2)}$. By definition,
\[
h(\mu_{a_0})=P(t_{a_0}f)-t_{a_0}\frac{h(\rho_1)}{h(\rho_2)}.
\]
Note that $\cal C(\rho_1,\rho_2)$ is the graph of a real analytic function $q$ defined implicitly as $P(-q(s)f)=s$. Setting $q(s)=-t_{a_0}$, it follows that 
\begin{align*}
M(\rho_1,\rho_2)=\frac{h(\mu_{a_0})}{h(\rho_1)}=\frac{P(-q(s)f)}{h(\rho_1)}+\frac{q(s)}{h(\rho_2)}=\frac{s}{h(\rho_1)}+\frac{q(s)}{h(\rho_2)}.
\end{align*}

Moreover, observe that
\[
1=\frac{d}{ds}P(-q(s)f)=\left(-\int f\mathrm{d}\mu_{-q(s)f}\right)q'(s).
\]
We conclude by recalling that
$\int f d\mu_{t_{a_0}f}=\frac{h(\rho_1)}{h(\rho_2)}$ and that the line passing through $(h(\rho_1),0)$ and $(0,h(\rho_2))$ has slope $-\frac{h(\rho_2)}{h(\rho_1)}$.
\end{proof}
\begin{rem}
It follows from strict convexity of the Manhattan curve and Theorem \ref{thm:Manhattan} that $M(\rho_1,\rho_2)\in (0,1)$. This fact is independently proved in Theorem \ref{thm:main}.
\end{rem}
\subsection{Decay of correlation number}\label{ssec:decay}
This section is dedicated to the proof of Theorem \ref{thm:pinching} from the introduction, which we restate here for the convenience of the reader. 

\medskip\noindent
{\textbf{Theorem} \ref{thm:pinching}\textbf{.}} {\em There exist sequences $(\rho_n)_{n=1}^\infty$ and $(\eta_n)_{n=1}^\infty$ in the Teichm{\"u}ller space $\cal T(S)$ such that the correlation numbers $M(\rho_n,\eta_n)$ satisfy}
\[
\lim_{n\to\infty} M(\rho_n,\eta_n)=0.
\]
\begin{proof} We construct two special families of hyperbolic structures $\rho_n$, $\eta_n$ and consider their corresponding geodesic currents $\nu_{\rho_n}$, $\nu_{\eta_n}$ as in section \ref{ssec:currents}. Our proof proceeds in two steps. First, we take geodesic currents $\nu_{\rho_n}+\nu_{\eta_n}$ given by the sum of the currents of $\rho_n$ and $\eta_n$ and show that their exponential growth rates satisfy $\lim\limits_{n\to\infty} h(\nu_{\rho_n}+\nu_{\eta_n})=0$. Then we show that this condition implies that the correlation number $M(\rho_n,\eta_n)$ goes to zero as well.
\smallskip

We consider two filling pair-of-pants decomposition $(\alpha_i)$ and $(\beta_i)$ on a hyperbolic structure $\rho_0$. A family of simple closed curves is filling if the complement of their union consists of topological discs. We take $(\rho_n)_{n=1}^\infty$ to be a sequence of hyperbolic structures obtained by pinching all $\alpha_i$ on $\rho_0$ so that the hyperbolic length $\ell_{\rho_n}(\alpha_i)=\epsilon_n$ with $\epsilon_n \to 0$ when $n\to \infty$. Similarly, we take $(\eta_n)$ to be another sequence of hyperbolic structures obtained by pinching all $\beta_i$ on $\rho_0$ so that the hyperbolic length $\ell_{\eta_n}(\beta_i)=\epsilon_n$. Note in such cases, we have that $\ell^H_{\rho_n}=\ell_{\rho_n}$ and $\ell^H_{\eta_n}=\ell_{\eta_n}$ and that the topological entropy $h(\rho_n)=h(\eta_n)=1$ for all $n$. We now proceed to prove $\lim\limits_{n\to\infty} h(\nu_{\rho_n}+\nu_{\eta_n})=0$.

By definition, the systole of $\nu_{\rho_n}+\nu_{\eta_n}$ is equal to $L_n=\inf\limits_{[\gamma]\in[\Gamma]}\left\{\ell^H_{\rho_n}([\gamma])+\ell^H_{\eta_n}([\gamma])\right\}$. 
Note that for a fixed $n$, $\#\{[\gamma]\in[\Gamma]\mid \ell^H_{\rho_n}([\gamma])+\ell^H_{\eta_n}([\gamma])<T\}<\infty$ for all $T>0$. In other words, the geodesic current $\nu_{\rho_n}+\nu_{\eta_n}$ is period minimizing and so we can apply Theorem \ref{thm:entropysystole} to find a constant $C$ depending only on the topology of $S$ such that
\[
h(\nu_{\rho_n}+\nu_{\eta_n})\leq \frac{C}{L_n}.
\]

Therefore to show $\lim\limits_{n\to\infty}h(\nu_{\rho_n}+\nu_{\eta_n})=0$, it suffices to show that $\lim\limits_{n\to\infty}L_n=\infty$. Because of the filling condition, the geodesic representative of any $[\gamma]\in [\Gamma]$ must intersect either curves in $(\alpha_i)$ or curves in $(\beta_i)$. By the Collar Lemma, each geodesic representative of $\alpha_i$ (resp. $\beta_i$) for $\rho_n$ (resp. $\eta_n$) is enclosed in a standard collar neighborhood of width approximately $\log\Big(\frac{1}{\epsilon_n}\Big)$.  In particular, every closed curve traverses a collar neighborhood of 
width approximately $\log\Big(\frac{1}{\epsilon_n}\Big)$ for the hyperbolic metric $\rho_n$ or the hyperbolic metric $\eta_n$ which implies that $\lim\limits_{n\to\infty}L_n=\infty$.

We now show that $\lim\limits_{n\to\infty} h(\nu_{\rho_n}+\nu_{\eta_n})=0$ implies that the correlation number goes to zero as well. Consider the reparametrization functions $f^{\eta_n}_{\rho_n}$ as in Lemma \ref{lem: rep} and notice that $\lambda(1+f^{\eta_n}_{\rho_n},\tau)=\ell^{H}_{\rho_n}([\gamma])+\ell^{H}_{\eta_n}([\gamma])$ for every periodic orbit corresponding to $[\gamma]\in[\Gamma]$. By Lemma \ref{lem:PressueZero}, for all $n$
\[
P(-h(\nu_{\rho_n}+\nu_{\eta_n})-h(\nu_{\rho_n}+\nu_{\eta_n})f^{\eta_n}_{\rho_n})=P(-h(\nu_{\rho_n}+\nu_{\eta_n})(1+f^{\eta_n}_{\rho_n}))=0.
\]
We have that $(h(\nu_{\rho_n}+\nu_{\eta_n}),h(\nu_{\rho_n}+\nu_{\eta_n}))\in\cal C(\rho_n,\eta_n)$ thanks to the characterization of the Manhattan curve in terms of the pressure function. If the line $y+x=2h(\nu_{\rho_n}+\nu_{\eta_n})$ is tangent to the Manhattan curve $\cal C(\rho_n,\eta_n)$, then by Theorem \ref{thm:Manhattan}, we obtain $M(\rho_n,\eta_n)=2h(\nu_{\rho_n}+\nu_{\eta_n})$. If not,  
then the line $y+x=2h(\nu_{\rho_n}+\nu_{\eta_n})$ must intersect $\cal C(\rho_n,\eta_n)$ at two points. See Figure \ref{fig:pinching}. Now by the mean value theorem and the convexity of the Manhattan curve, we have from Theorem \ref{thm:Manhattan} that there exists $0<a_n,b_n<2h(\nu_{\rho_n}+\nu_{\eta_n})$ such that
\[
0<M(\rho_n,\eta_n)=a_n+b_n<2h(\nu_{\rho_n}+\nu_{\eta_n}).
\]
This immediately implies that $M(\rho_n,\eta_n)$ goes to zero as $n$ goes to infinity.
\end{proof}
\begin{figure}[!htb]
    \centering
    \includegraphics[width=.5\textwidth]{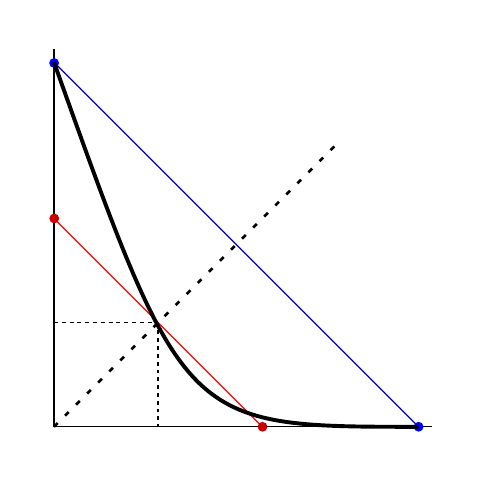}
    \put (-540,20){\makebox[0.7\textwidth][r]{\small{$0$}}}
    \put (-465,19){\makebox[0.7\textwidth][r]{\footnotesize{$h(\nu_{\rho_n}+\nu_{\eta_n})$}}}
    \put (-395,18){\makebox[0.7\textwidth][r]{\footnotesize{$2h(\nu_{\rho_n}+\nu_{\eta_n})$}}}
    \put (-542,127){\makebox[0.7\textwidth][r]{\footnotesize{$2h(\nu_{\rho_n}+\nu_{\eta_n})$}}}
    \put (-542,76){\makebox[0.7\textwidth][r]{\footnotesize{$h(\nu_{\rho_n}+\nu_{\eta_n})$}}}
    \put (-543,203){\makebox[0.7\textwidth][r]{\small{$1$}}}
    \put (-359,16){\makebox[0.7\textwidth][r]{\small{$1$}}}
    \caption{The Manhattan curve and the point $(h(\nu_{\rho_n}+\nu_{\eta_n}),h(\nu_{\rho_n}+\nu_{\eta_n}))$ described in the proof of Theorem \ref{thm:pinching}}
    \label{fig:pinching}
\end{figure}
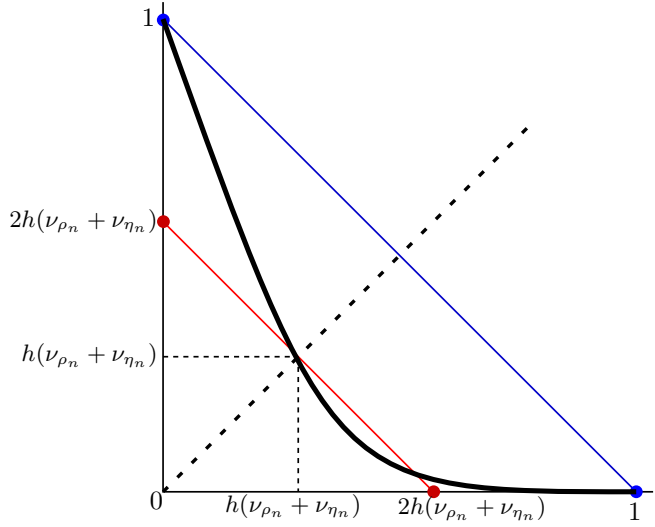
\section{The renormalized Hilbert length along cubic rays}\label{sec:applications}
\subsection{Cubic rays and affine spheres}
Building on work of Hitchin \cite{Hit_topol}, Labourie \cite{Labourie_FlatProjective} and Loftin \cite{Loftin_AffineSpheres} have independently parametrized $\frak C(S)$ as the vector bundle of holomorphic cubic differentials over $\cal T(S)$. We briefly recall this parametrization since we will use it in this subsection to study the correlation number. The reader can refer to \cite{Labourie_FlatProjective} and \cite{Loftin_AffineSpheres} for a detailed construction. 
This parametrization shows that the space of pairs formed by a complex structure $J$ on $S$ and a $J$-holomorphic cubic differential $q$ is in one-to-one correspondence with the space of convex real projective structures on $S$.

Because of the classical correspondence between hyperbolic structures and complex structures, we will sometimes blur the difference between a Riemann surface $X$ and a hyperbolic structure $X$. We recall that a holomorphic cubic differential $q$ is a holomorphic section of $K\otimes K \otimes K$ where $K$ is the canonical line bundle of a Riemann surface $X$. Locally, a holomorphic cubic differential can be written in complex coordinate charts as $q=q(z)dz\otimes dz \otimes dz$ (often written as $q(z)dz^3$) with $q(z)$ holomorphic. The cubic differential $q$ is invariant under change of coordinates, meaning that if we pick a different coordinate chart $w$, then $q(w)dw^3=q(z)dz^3$. The hyperbolic metric $\sigma$ that corresponds to the complex structure of $X$ induces a Hermitian metric $\langle\cdot,\cdot\rangle$ on the space of holomorphic cubic differentials on $X$. The {\em $L^2$-norm of a cubic differential $q$} is defined as
$$\|q\|_{X}=\int_{X}\langle q,q \rangle \rm{d}\sigma.$$

The correspondence between the space of pairs of complex structures on $S$ and holomorphic cubic differentials on associated Riemann surfaces with the set of convex real projective structures on $S$ goes through a geometrical object invariant under affine transformations, called the affine sphere.
We briefly describe the relation. Consider a pair $(J, q)$, where $J$ is a complex structure and $q$ is a holomorphic cubic differential on the Riemann surface $X_J$ associated to $J$. Then one can obtain an affine sphere, which is a hypersurface in $\mathbb{R}^3$ that is invariant under affine transformations, by solving a second order elliptic PDE called Wang's equation (see \cite[Section 4]{Loftin_AffineSpheres}). The affine sphere can be projected to $\mathbb{RP}^2$ so to obtain the developing image  $\Omega_{\rho}$ of some convex projective structure $\rho$. 
Conversely, given a convex projective structure $\rho$ and its developing image  $\Omega_{\rho}$, we can construct an affine sphere by solving a Monge-Amp{\'e}re equation (see \cite[Section 7]{Loftin_AffineSpheres}). The affine sphere will provide a complex structure $J$ and a holomorphic cubic differential $q$ on $X_J$.

Since hyperbolic structures correspond to complex structures, the above identification provides a way to parametrize the space of convex projective structures $\frak C(S)$ as the vector bundle of holomorphic cubic differentials over $\cal T(S)$. In particular,  we can fix a nonzero cubic differential $q$ and consider the associated family of representations $(\rho_t)_{\scriptscriptstyle t\in\bb R}$ in $\frak C(S)$ parametrized by $(tq)_{t\in\mathbb{R}}$.  When $t\geq 0$, we call such a family a \emph{cubic ray}. In this section, we study properties of the renormalized Hilbert length spectrum along cubic rays in $\frak C(S)$.

For the reminder of this section, when there is no ambiguity, we ease notation by writing $h_t=h(\rho_t)$ for the topological entropy, and $L^H_t=L^H_{\rho_t}$ for the renormalized Hilbert length spectrum of a cubic ray $(\rho_t)_{t\geq 0}$. With these notations, $\rho_0$ is the hyperbolic structure associated to $X_0$. The entropy $h_0$ equals $1$ and the renormalized Hilbert length spectrum $L^H_0$ is the $X_0$-length spectrum.

The following observation is important and follows immediately from work of Tholozan \cite{Tho_entropy}.

\begin{lem}\label{lem:tholo} Let $(\rho_t)_{t\geq 0}$ be a cubic ray. There exists $D>1$ and $t_0\geq0$ such that, for $t\geq t_0$,
\begin{align*}
\frac{1}{D}L^H_0 \leq L^H_t\leq DL^H_0.
\end{align*}
\end{lem}
\begin{proof} Tholozan \cite[Theorem 3.9]{Tho_entropy} proves that there exists $B>1$ and $t_0 \geq0$ such that, for $t\geq t_0$, 
\begin{equation}\label{eqn:tholo_length}
\frac{1}{B}t^{1/3}\ell^{H}_0\leq\ell^{H}_t\leq Bt^{1/3}\ell^{H}_0.
\end{equation}
In turn, this implies that
\[
\#\left\{[\gamma]\ \Big\vert\ \ell^{H}_0([\gamma])\leq \frac{1}{Bt^{1/3}}T\right\}\leq \#\{[\gamma]\mid \ell^{H}_t([\gamma])\leq T\}\leq \#\left\{[\gamma]\ \Big\vert\ \ell^{H}_0([\gamma])\leq \frac{B}{t^{1/3}}T \right\}.
\]
Since $h_0=1$, by Definition of topological entropy, the above inequalities imply that 
\begin{equation}\label{eqn:tholo_entropy}
\frac{1}{Bt^{1/3}}\leq h_t\leq \frac{B}{t^{1/3}}.
\end{equation}
The result follows by taking $D=B^2$ and combining the inequalities (\ref{eqn:tholo_length}) and (\ref{eqn:tholo_entropy}).
\end{proof}

The following example points out a problem with comparing un-renormalized length spectra and partially motivates our study of the renormalized Hilbert length spectrum. In particular, Equation (\ref{eq:badcorrelation}) below should be compared to the correlation theorem.

\begin{example}\label{ex:badcorrelation} Fix $\epsilon>0$. 
Let $\rho_0$ be a Fuchsian representation in $\mathfrak C(S)$ and consider a cubic ray $(\rho_t)_{t\geq 0}\subset \mathfrak C(S)$. Let $t_0\geq 0$ be as in Lemma \ref{lem:tholo} and consider any $t\geq t_0$.
Then
\begin{equation}\label{eq:badcorrelation}
\lim_{x\to\infty}\# \Big\{[\gamma]\in[\Gamma]\, \Big\vert\, \ell_0^{H}([\gamma]) \in \big(x,x+\eee\big),\  \ell_{t}^{H}([\gamma]) \in \big(x, x+\eee\big)\Big\}=0.
\end{equation}

\begin{proof}
For $t\geq t_0$, consider $\delta>0$ small enough so that $h_{t}<1< D-\delta$, where $D$ is as in Lemma \ref{lem:tholo}. There exists $M$ sufficiently large such that $\frac{Dx}{x+\eee}\geq D-\delta$ for all $x>M$. Therefore, for all $x>M$, if $x<\ell_0^H([\gamma])<x+\epsilon$, then
\begin{align*}
\ell^H_t([\gamma])\geq\frac{D}{h_t}\ell^H_{0}([\gamma]) >\frac{D}{D-\delta}x\geq x+\varepsilon.
\end{align*}
This shows that $\Big\{[\gamma]\in[\Gamma]\, \Big\vert\, \ell_0^{H}([\gamma]) \in \big(x,x+\eee\big),\  \ell_{t}^{H}([\gamma]) \in \big(x, x+\eee\big)\Big\}=\emptyset$ for all $x>M$, which implies that equation (\ref{eq:badcorrelation}) holds.
\end{proof}
\end{example}

\subsection{Cubic rays and correlation numbers} In contrast with Theorem \ref{thm:pinching}, we show some instances in which the correlation numbers $M(\rho_t, \eta_t)$ arising from two different cubic rays are uniformly bounded away from zero. We first introduce a convenient lemma that will be used in the proof of Theorems \ref{thm:differentfibers} and \ref{thm:cubicrays}.

\begin{lem}\label{lem:ManhattanEntropy}
Suppose $\rho$ and $\eta$ are different convex real projective structures such that $\rho^{*} \neq \eta$. Let $\upsilon_\rho$ and $\upsilon_\eta$ be the renormalized Hilbert currents associated to $\rho$ and $\eta$, respectively. For $s\in[0,1]$, let $u_s$ denote the geodesic current $s\upsilon_\rho+(1-s)\upsilon_\eta$. Then there exists a unique $s_0\in(0,1)$ such that
\[
M(\rho,\eta)=h(u_{s_0}).
\]
\end{lem}
\begin{proof}
Since $\eta\neq \rho,\rho^*$, the Manhattan curve $\cal C(\rho,\eta)$ is strictly convex. The line $\frac{x}{h(\rho)}+\frac{y}{h(\eta)}=1$ intersects the Manhattan curve in the first quadrant only at the points $(h(\rho),0)$ and $(0,h(\eta))$. It follows that for every $s\in[0,1]$ the straight line connecting $(sh(\rho),(1-s)h(\eta))$ and the origin must intersect $\mathcal{C}(\rho,\eta)$.
 By Lemma \ref{lem:PressueZero},
\[
P(-h(u_{s})sh(\rho)-h(u_{s})(1-s)h(\eta)f^{\eta}_{\rho})=P(-h(u_{s})(sh(\rho)+(1-s)h(\eta)f^{\eta}_{\rho}))=0,
\]
where $f^{\eta}_{\rho}$ is the reparametrization function from Lemma \ref{lem: rep}.
It follows from the definition of the Manhattan curve that every point on $\cal C(\rho,\eta)$ in the first quadrant can be written in the form of $(h(u_{s})sh(\rho),h(u_{s})(1-s)h(\eta))$ for some $s\in[0,1]$.
By Theorem \ref{thm:Manhattan}, there exists $s_0\in(0,1)$ such that
\[
M(\rho,\eta)=\frac{s_0h(u_{s_0})h(\rho)}{h(\rho)}+\frac{(1-s_0)h(u_{s_0})h(\eta)}{h(\eta)}=h(u_{s_0}).\qedhere
\]
\end{proof}

We first study the case of two cubic rays which lie in two different fibers of the vector bundle $\frak C(S)\to\cal T(S)$ given by the Labourie-Loftin parametrization of the space of convex projective surfaces.

\medskip\noindent
\textbf{Theorem} \ref{thm:differentfibers}. {\em Let $(\rho_t)_{t\geq 0}$, $(\eta_r)_{r\geq 0}$ be two cubic rays associated to two different hyperbolic structures $\rho_0\neq\eta_0$. Then, there exists a constant $C>0$ such that for all $t,r \geq 0$}
\[
M(\rho_t,\eta_r)\geq CM(\rho_0,\eta_0).
\]
\begin{proof} Fix $t, r\geq 0$ and write $\rho=\rho_t$ and $\eta=\eta_r$, for simplicity. Note that by hypothesis $L^H_\rho\neq L^H_\eta$ and $L^H_{\rho_0}\neq L^H_{\eta_0}$. For $s\in[0,1]$, let $u_s$ denote the geodesic current given by $
s \upsilon_\rho+(1-s)\upsilon_\eta$
where $\upsilon_\rho$ and $\upsilon_\eta$ are the renormalized Hilbert geodesic currents of $\rho$ and $\eta$, respectively. Denote by $h(u_s)$ the exponential growth rate for the geodesic current $u_s$.

By Lemma \ref{lem:tholo}, there exist constants $D_1,D_2$ depending on $\rho_0$ and $\eta_0$, respectively such that
\[
sL^H_\rho+(1-s)L^H_\eta\leq sD_1 L^H_{\rho_0}+(1-s)D_2 L^H_{\eta_0}\leq \max\{D_1,D_2\}( s L^H_{\rho_0}+(1-s) L^H_{\eta_0}).
\]
Set $C=\frac{1}{\max\{D_1,D_2\}}$ and $w_s=s\upsilon_{\rho_0}+(1-s)\upsilon_{\eta_0}$, where $\upsilon_{\rho_0}$ and $\upsilon_{\eta_0}$ are the renormalized Hilbert geodesic currents of $\rho_0$ and $\eta_0$, respectively. Then,
\begin{align*}
h(u_s)\geq C h(w_s).
\end{align*}
By Lemma \ref{lem:PressueZero}, the intersection between the Manhattan curve $\cal C(\rho_0, \eta_0)$ and the line passing through the origin and the point $(s,(1-s))$ has coordinates $(s h(w_s),(1-s)h(w_s))$ and lies on the line $y+x=h(w_s)$. Then, by Theorem \ref{thm:Manhattan}, 
\[
h(w_s)\geq M(\rho_0,\eta_0).
\]
See Figure \ref{fig:h(ws)}. Finally, by Lemma \ref{lem:ManhattanEntropy} there exists $s_0\in(0,1)$ such that
\[
M(\rho,\eta)=h(u_{s_0})\geq h(w_{s_0})\geq C M(\rho_0,\eta_0)
\]
which concludes the proof.
\end{proof}
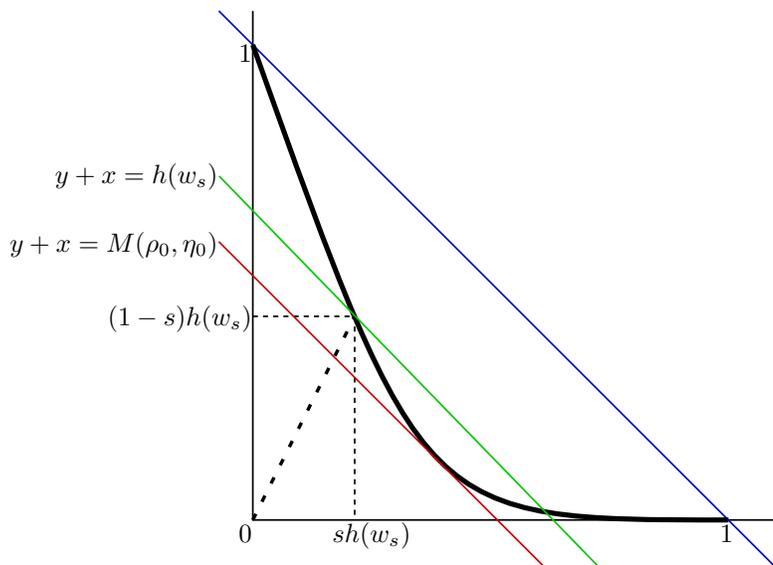
\begin{figure}[!htb]
    \centering
    \includegraphics[width=.45\textwidth]{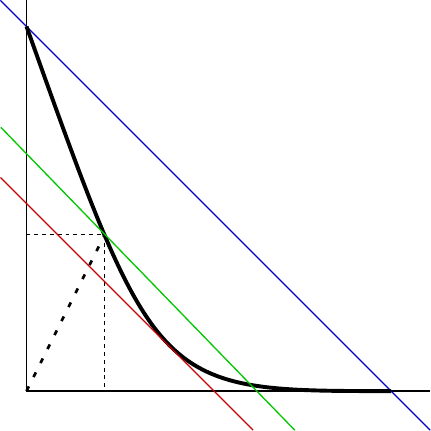}
    \put (-530,11){\makebox[0.7\textwidth][r]{\small{$0$}}}
    \put (-348,11){\makebox[0.7\textwidth][r]{\small{$1$}}}
    \put (-530,193){\makebox[0.7\textwidth][r]{\small{$1$}}}
    \put (-470,11){\makebox[0.7\textwidth][r]{\small{$s h(w_s)$}}}
    \put (-530,93){\makebox[0.7\textwidth][r]{\small{$(1-s)h(w_s)$}}}
    \put (-543,121){\makebox[0.7\textwidth][r]{\small{$y+x=M(\rho_0,\eta_0)$}}}
    \put (-543,147){\makebox[0.7\textwidth][r]{\small{$y+x=h(w_s)$}}}
    \caption{The correlation number of $\rho_0$ and $\eta_0$ is less or equal to the exponential growth rate of the geodesic current $w_s$}
    \label{fig:h(ws)}
\end{figure} 

Finally, we prove Theorem \ref{thm:cubicrays} from the introduction.

\medskip\noindent
{\textbf{Theorem} \ref{thm:cubicrays}\textbf{.}} {\em Let $\rho_t$ and $\eta_t$ be two cubic rays associated to two different holomorphic cubic differentials $q_1$ and $q_2$ on a hyperbolic structure $X_0$ such that $q_1,q_2$ have unit $L^2$-norm with respect to $X_0$ and $q_1\neq -q_2$. Then the correlation number $M(\rho_{t},\eta_{t})$ is uniformly bounded away from zero as $t$ goes to infinity.}
\begin{proof} 
We note that $\rho_0=\eta_0$ and $L_{\rho_0}^{H}=L_{\eta_0}^{H}=L_{0}^{H}$. Recall from Lemma \ref{lem:tholo}, there exists $D_1,D_2>1$ such that for $t$ large,
\begin{align*}
\frac{1}{D_1}L^H_0 \leq L^H_{\rho_t}\leq D_1L^H_0\qquad\text{ and }\qquad \frac{1}{D_2}L^H_0 \leq L^H_{\eta_t}\leq D_2L^H_0.
\end{align*}

Consider the renormalized Hilbert currents $\upsilon_{\rho_t}$ and $\upsilon_{\eta_t}$ for $\rho_{t}$ and $\eta_{t}$, respectively. Let $s\in[0,1]$ and set the geodesic current $u_{t,s}=s\upsilon_{\rho_t}+(1-s)\upsilon_{\eta_t}$. 
As a first step, we show that the entropy of the geodesic current $u_{t,s}=s\upsilon_{\rho_t}+(1-s)\upsilon_{\eta_t}$ is uniformly bounded away from zero. We have from Lemma \ref{lem:tholo} for $t$ large, for all $[\gamma]\in[\Gamma]$, 
\[
s L_{\rho_t}^H([\gamma])+(1-s)L^H_{\eta_t}([\gamma])\leq (sD_1+(1-s)D_2)L_0^H([\gamma])\leq \max\{D_1, D_2\} L^H_0([\gamma]).
 \]
Thus, we obtain $h(u_{t,s}) \geq \frac{1}{D}$ where $D=\max\{D_1,D_2\}$.

Next we want to show this implies the correlation number $M(\rho_{t},\eta_{t})$ is bounded away from zero.
Because $q_2\neq -q_1$, we know $\rho_t^{*} \neq \eta_t$ for $t>0$ (see \cite[section 5 and section 8]{Loftin_AffineSurvey}). For each $t>0$, by Lemma \ref{lem:ManhattanEntropy}, we can find some $s\in(0,1)$ such that 
\[
M(\rho_t,\eta_t)=h(u_{t,s}).
\]
We conclude that $M(\rho_{t},\eta_{t}) \geq \frac{1}{D}$ for $t$ large, as desired.
\end{proof}

\subsection{Cubic rays and geodesic currents}\label{ssec:systole}
In this section we observe two properties of renormalized Hilbert currents  $\upsilon_{\rho_t}=h_t\nu_{\rho_t}$ along cubic rays $(\rho_t)_{t\geq 0}$ which readily follow from Lemma \ref{lem:tholo}. Let us start by showing that the self-intersection of the renormalized Hilbert geodesic current is uniformly bounded along cubic rays.
\begin{prop} There exists $C>1$ such that for all $t\in\bb R$
\[
\frac{1}{C}\leq i(\upsilon_{\rho_t},\upsilon_{\rho_t})\leq C
\]
where $\upsilon_{\rho_t}$ is the renormalized Hilbert geodesic current of $\rho_t$.
\end{prop}
\begin{proof} { Recall from Section \ref{ssec:currents} that Bonahon proved that $i(\upsilon_{\rho_0},\upsilon_{\rho_0})=-\pi^2\chi(S)$, where $\chi(S)$ is the Euler characteristic of $S$.} Corollary 5.2 in \cite{BCLS_currents} states that
\[
\left(\inf_{[\gamma]\in[\Gamma]}\frac{L^H_t([\gamma])}{L^H_0([\gamma])}\right)^2\leq \frac{i(\upsilon_{\rho_t},\upsilon_{\rho_t})}{i(\upsilon_{\rho_0},\upsilon_{\rho_0})}\leq \left(\sup_{[\gamma]\in[\Gamma]}\frac{L^H_t([\gamma])}{L^H_0([\gamma])}\right)^2
\]
so
\[
-\pi^2\chi(S)\left(\inf_{[\gamma]\in[\Gamma]}\frac{L^H_t([\gamma])}{L^H_0([\gamma])}\right)^2\leq i(\upsilon_{\rho_t},\upsilon_{\rho_t})\leq-\pi^2\chi(S) \left(\sup_{[\gamma]\in[\Gamma]}\frac{L^H_t([\gamma])}{L^H_0([\gamma])}\right)^2
\]
We conclude by applying Lemma \ref{lem:tholo} for $t$ large.
\end{proof}
Bonahon \cite{Bon_currents} showed that Thurston's compactification of the Teichm\"uller space can be understood via geodesic currents. Explicitly, given a diverging sequence of hyperbolic structures $m_t$, there exists $\lambda_t>0$ and a geodesic current $\alpha$ with $i(\alpha,\alpha)=0$, a {\em measured lamination}, such that $\lambda_t\nu_{m_t}$ converges up to subsequences to $\alpha$. In this case, the systole of $\alpha$ vanishes, i.e.
\[
\rm{Sys}(\alpha)=\inf_{[\gamma]\in[\Gamma]}i(\alpha,\delta_\gamma)=0.
\]
Burger, Iozzi, Parreau, and Pozzetti \cite{BIPP_compact} show that there exist diverging sequences of Hilbert geodesic currents which converge projectively to a current $\alpha$ with $\rm{Sys}(\alpha)>0$. We use Lemma \ref{lem:tholo} to show that this happens along cubic rays.

\medskip\noindent
{\textbf{Theorem} \ref{prop:degen}\textbf{.}} {\em As $t$ goes to infinity, the renormalized Hilbert geodesic current $\upsilon_{\rho_t}$ along a cubic ray converges, up to passing to a subsequence, to a geodesic current $\upsilon$ with $\rm{Sys}(\upsilon)>0$.}
\begin{proof} Suppose $(\alpha_i)$ and $(\beta_j)$ are filling pair-of-pants decompositions, i.e. they are pants decompositions such that the complement of their union is a collection of topological discs. Set $u=\sum_{\gamma\in (\alpha_i)\cup(\beta_j)}\delta_\gamma$ and let $M=\max_{\gamma\in(\alpha_i)\cup(\beta_j)}L^H_0([\gamma])$.
For every $t\in\bb R$, we have
\[
i(\upsilon_{\rho_t},u)=\sum_{\gamma\in(\alpha_i)\cup(\beta_j)}L^H_t([\gamma]).
\]
By Lemma \ref{lem:tholo} for $t$ large $\upsilon_t$ lies in $
\left\{\nu\in\cal C(S)\mid i\left(\nu,u\right)\leq (6g-6)DM \right\}$ for some $D>1$. This set is compact by \cite[Proposition 4]{Bon_currents} and by linearity of the intersection number. Thus, $\upsilon_t$ converges, up to passing to a subsequence, to $\upsilon\in\cal C(S)$. Applying Lemma \ref{lem:tholo} again, we see that for all $t$ large $\rm{Sys}(\upsilon_{\rho_t})$ is greater or equal to $D^{-1}\rm{Sys}(\upsilon_{\rho_0})>0$. By continuity, the systole $\rm{Sys}(\upsilon)$ is strictly positive. 
\end{proof}

\section{Generalization to Hitchin representations}\label{sec:Hitchin}
In this section, we illustrate how to generalize the main results of sections \ref{sec:corrthm} and \ref{sec:corrnum} to the context of \emph{Hitchin representations}. 

We start with introducing Hitchin components and Hitchin representations. Given $\rho\in\mathcal{T}(S)$, we can postcompose the corresponding holonomy representation $\rho\colon\Gamma\to \rm{PSL}(2,\bb R)$ with the unique (up to conjugation) irreducible representation $i\colon\rm{PSL}(2,\bb{R}) \to \rm{PSL}(d,\bb{R})$ given by the action of $A\in\rm{PSL}(2,\bb{R})$ on the space of degree $d-1$ homogeneous polynomials in two variables. The \emph{Hitchin component} $\cal H_d(S)$ is the connected component of the character variety $\sf{Hom}(\Gamma,\rm{PSL}(d,\bb R))/\!/\rm{PSL}(d,\bb R)$ containing $i\circ\rho$. The copy of the Teichm\"uller space $\cal T(S)$ embedded in the Hitchin component is its {\em Fuchsian locus}. Hitchin proves in \cite{Hit_topol}, using Higgs bundles techniques, that $\cal H_d(S)$ is homeomorphic to an open Euclidean ball of dimension $(2g-2)\cdot\dim(\rm{PSL}(d,\mathbb{R}))$.  When $d=2$, the Hitchin component $\cal H_2(S)$ coincides with the Teichm\"uller space $\cal T(S)$. Choi and Goldman \cite{CG_convex} identify the Hitchin component $\cal H_3(S)$ with the space of convex projective structures on the surface $S$, which is the main focus of Sections \ref{sec:prelim} through \ref{sec:applications}.

In this section, we establish the correlation theorem \ref{thm:HitchinMain} for pairs of Hitchin representations for any $d\geq 3$. In this setting, the length spectra will not be defined geometrically as for the Hilbert length spectrum of a convex projective structure, but they will be interpreted as periods of the {\em Busemann cocycle}, which was first introduced by Quint \cite{Quint_cocycle}. Then, we follow a construction of Sambarino \cite{Sambarino_Quantitative, Sambarino_orbital} to replace the geodesic flow on the unit tangent bundle of a convex projective surface with Sambarino's {\em translation flows}. These flows are not necessarily Anosov, but fit in more general framework of \emph{metric Anosov flows}. Nevertheless, the main results coming from Thermodynamic formalism needed for this paper still hold in this setting (\cite[Section 3]{Pollicot_SmaleFlow}).

\subsection{Length functions, Busemann cocycles and entropy}

In this section, we define length functions for Hitchin representations and the Busemann cocycles associated to them. We refer to \cite{Quint_cocycle, Sambarino_Quantitative} for a more detailed construction.
 
We need to recall some Lie theory. Let $\rm{G}=\rm{PSL}(d,\bb R)$ and consider the standard choices of Cartan subspace
\[
\frak a=\{\vec x\in\bb R^d\mid x_1+\dots+x_d=0\}
\]
and positive Weyl chamber $\frak a^+=\{\vec x\in\frak a\mid x_1\geq \dots\geq x_d\}$.  Let $\lambda\colon\rm G\to\frak a^+$ be the \emph{Jordan projection}
\[
\lambda(g)=(\log \lambda_{1}(g),\dots, \log \lambda_{d}(g))
\]
consisting of the logarithms of the moduli of the eigenvalues of $g$ in nonincreasing order. We will use the following fundamental property of the Jordan projection of the image of a Hitchin representation. 
\begin{thm}[Fock-Goncharov \cite{FG}, Labourie \cite{Lab_Anosov}]\label{thm:loxodromic} Every Hitchin representation $\rho\in\mathcal H_d(S)$ is discrete and faithful and for every $[\gamma]\in[\Gamma]$, 
\[
\lambda_1(\rho(\gamma))>\dots>\lambda_d(\rho(\gamma))>0.
\] 
\end{thm}

The \emph{limit cone} $\mathcal{L}_{\rho}$ of a Hitchin representation $\rho\in\mathcal H_d(S)$, introduced in \cite{Benoist_propasy}, is the closed cone of $\frak a^+$ generated by $\lambda(\rho(\gamma))$. This cone contains all the rays spanned by positive multiples of $\lambda(\rho(\gamma))$ for all $\gamma\in \Gamma$. The (positive) \emph{dual cone} of  $\mathcal{L}_{\rho}$ is defined as $\mathcal{L}_{\rho}^{*}=\{\phi\in\frak a^{*}: \phi|_{\mathcal{L}_{\rho}} \geq 0\}$. For every $i=1,\dots, d-1$, the \emph{$i$-th simple root} $\alpha_i\colon\frak a\to \bb R$, defined by $\alpha_i(\vec x)=x_i-x_{i+1}$, is an important example of element in the interior of $\mathcal{L}_{\rho}^{*}$. We will focus on linear functionals in the set
\[
\Delta=\left\{c_1\alpha_1+\dots+c_{d-1}\alpha_{d-1} \mid c_i\geq 0,\  \sum_ic_i>0\right\}
\]
which is contained in the interior of $\mathcal{L}_{\rho}^{*}$ by Theorem \ref{thm:loxodromic}.

Given $\phi\in\mathcal{L}_{\rho}^{*}$, the {\em $\phi$-length of $[\gamma]\in[\Gamma]$} is defined as 
\[
\ell^\phi_\rho([\gamma])=\phi(\lambda(\rho(\gamma)))>0,
\]
and its \emph{exponential growth rate} is 
\[
h^{\phi}(\rho)={\limsup_{T\to\infty} \frac{1}{T}\#\log\{[\gamma]\in[\Gamma]\mid \ell^\phi_\rho([\gamma])\leq T\}.}
\]
An important property of linear functionals in $\Delta$ is the following.
\begin{lem}
If $\phi\in\Delta$, then $h^{\phi}(\rho)\in(0,\infty)$.
\end{lem}
\begin{proof} Since we know 
$h^{\alpha_i}(\rho)=1$ from \cite[Theorem B]{Pollicot_Zeta}, by Lemma 2.7 in \cite{PotrieSambarino}, a linear functional $\varphi\in\mathcal{L}_{\rho}^*$ has finite and positive entropy if and only if it belongs to the interior of $\mathcal{L}_{\rho}^{*}$. Since elements in $\Delta$ are contained in the interior of $\mathcal{L}_{\rho}^{*}$, the entropy $h^{\phi}(\rho)$ is positive and finite for every $\phi\in\Delta$.
\end{proof}

The $\phi$-length function can be realized via the period of the \emph{$\phi$-Busemann cocycle}. To define \emph{$\phi$-Busemann cocycles}, we need to introduce the \emph{Frenet curve} $\xi_\rho$ for a Hitchin  representation $\rho$. Denote by $\sf F_d$ the space of complete flags in $\bb R^d$. Fock and Goncharov \cite{FG} and Labourie \cite{Lab_Anosov} show that for every Hitchin representation $\rho$, there exists a unique (up to conjugation) $\rho$-equivariant bi-H\"older continuous {\em Frenet curve} $\xi_\rho\colon\partial\Gamma\to \sf F_d$ which is transverse, positive, and it satisfies certain contraction properties. Moreover, the Frenet curve is \emph{dynamics preserving}:
if $\gamma^+\in\partial\Gamma$ is the attracting fixed point of $\gamma\in\Gamma$, then $\xi_\rho(\gamma^+)$ is the attracting eigenflag of $\rho(\gamma)$. We refer to \cite[Thm 1.15]{FG} and \cite[Thm 4.1]{Lab_Anosov} for precise statements.

We are now ready to define $\phi$-Busemann cocycles for $\phi\in\Delta$. Fix $F_0\in\sf F_d$ and observe that for every $F\in\sf F_d$, there exists $k\in\rm{SO}(d)$ such that $F=kF_0$. Every $M\in \rm{SL}(d,\bb R)$ can be written as $M=L(\exp{\sigma(M)})U$, for some $L\in\rm{SO}(d)$, $\sigma(M)\in\frak a$ and $U$ is unipotent and upper triangular. This is known as the Iwasawa decomposition of $M$. The \emph{Iwasawa cocycle} $B\colon \rm{SL}(d,\bb R)\times \sf F_d\to\frak a$ is defined as $B(A,F)=\sigma(Ak)$.

The \emph{vector valued Busemann cocycle} $B_\rho\colon\Gamma\times\partial\Gamma\to\frak a$ of a Hitchin representation $\rho$ with Frenet curve $\xi_\rho$ is
\[
B_\rho(\gamma,x)=B(\rho(\gamma),\xi_\rho(x)).
\]
For every linear functional $\phi\in\Delta$ we set $B^\phi_\rho(\gamma,x)=\phi(B_\rho(\gamma,x))$.
Lemma 7.5 in \cite{Sambarino_Quantitative} directly shows that the cocycle $B^{\phi}_{\rho}$ encodes the $\phi$-length via the equality $B_\rho^\phi(\gamma,\gamma^+)=\ell^\phi_\rho([\gamma])$, for every $\gamma\in\Gamma$.
\smallskip

The $\phi$-Busemann cocycle is an example of a H\"older cocycle.
\begin{defn}\label{def:cocycle} A \emph{H\"older cocycle} is a map $c\colon\Gamma\times\partial\Gamma\to \bb R$ such that
\[
c(\gamma\eta,x)=c(\gamma,\eta x)+c(\eta,x)
\]
for any $\gamma,\eta\in\Gamma$ and $x\in\partial\Gamma$ and there exists $\alpha\in(0,1)$ such that $c(\gamma,\cdot)$ is $\alpha$-H{\"o}lder for all $\gamma\in\Gamma$. 
\end{defn}
For a general H\"older cocycle, let $\ell_c(\gamma)=c(\gamma,\gamma^+)$ denote the $c$-length (also known as {\em period}) of $\gamma\in\Gamma$. Two H\"older cocycles are \emph{cohomologous} if they have the same periods. We can define the \emph{exponential growth rate} $h_c$ of a general H{\"o}lder cocycle $c:\Gamma\times\partial\Gamma \to \mathbb{R}$ as
\[
h_c=\limsup_{T\to\infty} \frac{1}{T}\#\log\{[\gamma]\in[\Gamma]\mid \ell_c(\gamma)\leq T\}.
\]
Note that the exponential growth rate $h^{\phi}(\rho)$ of the linear functional $\phi\in \Delta$ is the same as the exponential growth rate $h_{B^{\phi}_{\rho}}$ of the cocycle $B^{\phi}_{\rho}$. The number $h^\phi(\rho)$ is also called the \emph{topological entropy} of a Hitchin representation $\rho$ with respect to the $\phi$-length. Indeed, $h^\phi(\rho)$ is the topological entropy of a {\em translation flow} as we recall in the next section.

\subsection{Translation flows and metric Anosov flows}
In this subsection, we recall how to equip Hitchin representations with translation flows which will allow us to use thermodynamic formalism tools to study them. We start by recalling the construction of Sambarino's translation flow \cite{Sambarino_Quantitative} which is analogous to Hopf's parametrization of the geodesic flow of a negatively curved manifold.

Let $\partial^2\Gamma=\{(x,y)\in\partial \Gamma \times \partial \Gamma: x\neq y\}$. The \emph{translation flow} $\{\varphi_t\}_{t\in\bb R}$ on $\partial^{2}\Gamma \times \mathbb{R}$ is 
$$\varphi_t(x,y,s)=(x,y,s-t).$$
Given a H\"older cocycle $c\colon\Gamma\times\partial\Gamma\to \bb R$, the group $\Gamma$ acts on $\partial^{2}\Gamma \times \mathbb{R}$ by 
$$\gamma(x,y,t)=(\gamma x, \gamma y, t-c(\gamma, y) ).$$
The translation flow $\{\varphi_t\}_{t\in\bb R}$ then descends to the quotient $M_c$ of $\partial^2 \Gamma \times \mathbb{R}$ by the $\Gamma$-action via $c$. 
Sambarino proves a reparametrization theorem for the translation flow $\varphi_t$.

 First, recall that two (H{\"o}lder) continuous flows $\psi_t$ and $\psi_t'$ on compact metric spaces $X$ and $Y$, respectively, are {\em (H\"older) conjugated} if there exists a (bi-H\"older) homeomorphism $h\colon X\to Y$ such that $\psi_t'=h\circ \psi_t \circ h^{-1}$ for every $t\in\mathbb R$.

\begin{thm}[{Sambarino \cite[Thm 3.2]{Sambarino_Quantitative}}]\label{thm:SambarinoReparametrization}
Given a H{\"o}lder cocycle $c$ with non-negative periods and $h_c\in(0,\infty)$, the action defined by the cocycle $c$ on $\partial^2 \Gamma \times \mathbb{R}$ is proper and co-compact. Moreover, the translation flow $\varphi_t$ on the quotient space $M_c$ is H\"older conjugated to a H{\"o}lder reparametrization of the geodesic flow on $T^{1}X_{0}$, where $X_0$ is a(ny) hyperbolic structure on $S$. The translation flow on $M_c$ is topologically mixing and its topological entropy equals $h_c$.
\end{thm}

From now on we will always assume that the cocycle $c$ is of the form $B^\phi_\rho$ for some Hitchin representation $\rho$ and $\phi\in\Delta$. In this case, the hypotheses of Theorem \ref{thm:SambarinoReparametrization} are satisfied. We will still denote the translation flow associated to $B^\phi_\rho$ by $\varphi_t$ and write $M_\rho^\phi=M_{B^\phi_\rho}$.

We want to use Thermodynamic formalism methods for the translation flow $\varphi_t$ and we briefly explain why we can do so in this setting. While it is well known that a H{\"o}lder reparametrization of an Anosov flow is an Anosov flow \cite[Page 122]{AnosovSinai}, H{\"o}lder conjugacy does not necessarily preserve the Anosov property \cite{Gogolev}. In order to construct symbolic codings and use  results from the Thermodynamic formalism for Hitchin representations, we will work in the more general setting of metric Anosov flows. \emph{Metric Anosov flows (or Smale flows)} were introduced by Pollicott in \cite{Pollicot_SmaleFlow} to generalize classical results for Anosov flows and Axiom A flows. In particular, he constructed a symbolic coding for metric Anosov flows \cite{Pollicot_SmaleFlow}.

Our definition of metric Anosov flow, which is better suited to our purposes, differs slightly from the original definition in \cite{Pollicot_SmaleFlow}. Recall for any continuous flow $\psi_t$ on a compact metric space $X$, the \emph{local stable set} of a point $x\in X$ is defined for $\epsilon>0$ as
\[
W^s_{\epsilon}(x)=\{y\in X: d(\psi_t x , \psi_t y) \leq \epsilon,\, \text{  }\forall t \geq 0 \textnormal{ and }  d(\psi_t x, \psi_t y) \to 0 \textnormal{ as } t \to \infty\}
\]
The \emph{local unstable set} of a point $x$ for $\epsilon>0$ is
\[
W^u_{\epsilon}(x)=\{y\in X: d(\psi_{-t} x , \psi_{-t} y) \leq \epsilon,\, \text{  }\forall t \geq 0 \textnormal{ and }  d(\psi_{-t} x, \psi_{-t} y) \to 0 \textnormal{ as } t \to \infty\}
\]
\begin{defn}\label{def, MetricAnosov}
A continuous flow $\psi$ on a compact metric space $X$ is \emph{metric Anosov} if 
\begin{enumerate}
    \item  There exist positive $C, \lambda, \epsilon$ and $\alpha\in(0,1]$ such that 
    $$d(\psi_tx, \psi_ty)\leq C e^{-\lambda t}d(x,y)^{\alpha} \text{   when $y\in W^{s}_{\epsilon}(x)$ and $t\geq 0$}$$
and    
 $$d(\psi_{-t}x, \psi_{-t}y)\leq C e^{-\lambda t}d(x,y)^{\alpha} \text{   when $y\in W^{u}_{\epsilon}(x)$ and $t\geq 0$}.$$
 \item There exists $\delta>0$ and a continuous map $\upsilon$ on the set $\{(x,y)\in X\times X: d(x,y) \leq \delta \}$ such that $\upsilon=\upsilon(x,y)$ is the unique value for which  $W^u_{\epsilon}(\psi_{\upsilon}x) \cap W^s_{\epsilon}(y)$ is nonempty.  The set $W^u_{\epsilon}(\psi_{\upsilon}x) \cap W^s_{\epsilon}(y)$ consists of a single point, denoted as $\braket{x,y}$.
\end{enumerate}
\end{defn}

\begin{rem}\label{ConjugacyPreservesMetricAnosov}
 Metric Anosov flows in Definition \ref{def, MetricAnosov} have a symbolic coding. Indeed, these flows verify one of the key properties (\cite[Lemma 1.5]{PeriodicOrbits-Bowen}) used by Bowen to build symbolic codings for Axiom A flows. From this, and after adapting Bowen's arguments to accommodate for the exponent $\alpha$ in the definition, we observe that the metric Anosov flows in Definition \ref{def, MetricAnosov} satisfy the expansivity, tracing and specification properties \cite[Prop 1.6 and Section 2]{PeriodicOrbits-Bowen}. These are also the crucial properties used by Pollicott in his construction of symbolic codings for Smale flows \cite{Pollicot_SmaleFlow}.
\end{rem}

\begin{prop} \label{prop,ReparConjPreserveMetriAnosov}
Let $\psi^1$ be a H{\"o}lder continuous metric Anosov flow on a compact metric space $X$. 
If $\psi^2$ is a flow on a compact metric space $Y$ and $\psi^2$ is H\"older conjugate to $\psi^1$, then $\psi^2$ is a H{\"o}lder continuous metric Anosov flow.
\end{prop}

\begin{proof}

We show here that  a H{\"o}lder conjugacy preserve metric Anosov properties.  
By hypothesis, there exists a bi-H{\"o}lder homeomorphism $h:X\to Y$ with H{\"o}lder exponent $\alpha_0\in(0,1]$ (for both $h$ and $h^{-1}$) such that $h\circ \psi_t^1= \psi_t^2\circ h$.  We want to show that $\psi_t^2$ also satisfies the metric Anosov property.

Given $x_2\in Y$, we want to find suitable parameters so that conditions (1) and (2) in  Definition \ref{def, MetricAnosov} hold. We show condition (1) for local stable sets. Suppose $x_2=h(x_1)$.  Because $\psi_t^1$ is metric Anosov, we can find $\varepsilon_1, C_1, \lambda_1, \delta_1$ positive and $\alpha_1\in(0,1]$ so that
\[
W^s_{\epsilon_1}(x_1)=\{y_1\in X: d_X(\psi^1
_t x_1,\psi^1_t y_1)\leq \epsilon_1 \text{ and } d_X(\psi^1_t x_1 , \psi^1_t y_1) \leq C_1 e^{-\lambda_1 t} d_X(x_1,y_1)^{\alpha_1}, \text{  } \forall t \geq 0 \}.
\]
Note
\begin{align*}
    d_Y(\psi^2_th(x_1), \psi^2_t h(y_1))&=d_Y(h(\psi_t^1x_1), h(\psi_t^1 y_1))\leq C_h d_X(\psi^1_tx_1,\psi^1_ty_1)^{\alpha_0}\\
    &\leq C_h(C_1 e^{-\lambda_1t} d_X(x_1,y_1)^{\alpha_1})^{\alpha_0}
    \leq C_2 e^{-\lambda_1 \alpha_0 t}d_Y(h (x_1), h (y_1))^{\alpha_0^2\alpha_1} 
\end{align*}
Here we have used the H{\"older} properties of both $h$ and $h^{-1}$ and $C_h$ is a constant from the H{\"older} property of $h$. This implies that we can find $\varepsilon_2 >0$ so that $W^{s}_{\varepsilon_2}(x_2)$ satisfies condition (1) in Definition \ref{def, MetricAnosov} for $\psi^2_t$ with positive $C_2, \lambda_2=\lambda_1 \alpha_0$ and $\alpha_2= \alpha_0^2\alpha_1$. The argument is similar for local unstable sets. Furthermore, given $x_2, y_2\in Y$ close enough, condition (2) in  Definition \ref{def, MetricAnosov} can be verified by taking $\langle x_2,y_2\rangle= h(\langle h^{-1} (x_2), h^{-1} (y_2) \rangle)$. We conclude that $\psi^2_t$ is metric Anosov.
\end{proof}

We now show that the translation flow is a metric Anosov flow.
\begin{prop}\label{translationMetricAnosov}
The translation flow $\varphi_t$ on $M^\phi_{\rho}$ is a metric Anosov flow which admits a symbolic coding with H{\"o}lder continuous roof functions.
\end{prop}
\begin{proof}
Let $X_0$ be a base hyperbolic surface. The geodesic flow on $T^{1}X_{0}$ is a (metric) Anosov flow which admits a symbolic coding with H{\"o}lder roof function. From Theorem \ref{thm:SambarinoReparametrization}, we know that the translation flow $\varphi_t$ associated to the cocycle $B_{\rho}^{\phi}$ defined on the quotient space $M^\phi_{\rho}$ is H{\"o}lder conjugate to a H{\"o}lder reparametrization of the geodesic flow $\psi_t$ on $T^{1}X_{0}$.  In other words, there exists a H\"older continuous function $f\colon T^1X_0\to \bb R_{\geq 0}$ such that $(h^{-1}\circ\varphi\circ h)_{t}=\psi^f_t$ for all $t\in \bb R$. Since $\psi^f_t$ is an Anosov flow (see \cite[Page 122]{AnosovSinai}), $\phi_t$ is a metric Anosov flow by Proposition \ref{prop,ReparConjPreserveMetriAnosov}.

The coding is therefore preserved. The roof function remains H{\"o}lder by either H{\"o}lder conjugacy or H{\"o}lder reparametrization as it is given by an opportune composition of H\"older functions.  The translation flow is therefore a metric Anosov flow on a compact metric space that admits a symbolic coding with H{\"o}lder roof function.
\end{proof}

\subsection{Reparametrization functions and independence lemma}
After constructing flows for Hitchin representations, the next thing that we want to do is to define the \emph{reparametrization function} in the setting of Hitchin representations. In Section \ref{ssec:reparam}, given a positive H\"older continuous function $f$ on the unit tangent bundle $T^1X_{\rho}$ of a convex real projective structure, we defined a reparametrization of the flow by time change. This can be done more generally for any H\"older continuous flow on a compact metric space $X$. In particular, given two Hitchin representations $\rho_1$ and $\rho_2$ in $\cal H_{d}(S)$, we will show the existence of a positive H\"older continuous \emph{reparametrization function} $f_{\rho_1}^{\rho_2}: M^\phi_{\rho_1} \to \mathbb{R}_{> 0}$ that encodes the $\phi$-length spectrum of $\rho_2$. 

We start with a lemma that relates the positivity of the H{\"o}lder reparametrization function to the positivity of its entropy. The \emph{entropy} of a H{\"o}lder function $f$ on a compact metric space $X$ equipped with a metric Anosov flow is defined as 
\[
h_f=\limsup_{T\to\infty} \frac{1}{T}\#\log\left\{
\tau \text{ periodic} \mid \int_{\tau} f\leq T\right\}.
\]

\begin{lem}[{Ledrappier \cite[Lemma 1]{Ledrappier}, Sambarino \cite[Lemma 3.8]{Sambarino_Quantitative}}]\label{lem:PosEntropy}
Let $f:X\to \mathbb{R}$ be a H{\"o}lder continuous function  with non-negative periods on a compact metric space $X$ equipped with a topological transitive  metric Anosov flow. The following are equivalent:
\begin{enumerate}
    \item the function $f$ is cohomologous to a positive H{\"o}lder continuous function.
    \item there exists $\kappa>0$ such that $\int_{\tau} f>\kappa p(\tau)$ where $p(\tau)$ is the period of $\tau$.
    \item the entropy $h_f\in (0,\infty)$. 
\end{enumerate}
\end{lem}

We also need the following theorem of Ledrappier \cite{Ledrappier} which establishes the correspondence between cohomologous H{\"o}lder cocycles and cohomologous H{\"o}lder continuous functions on $T^1X_0$. We will state this theorem for vector valued H\"older cocycles. Their definition is analogous to Definition \ref{def:cocycle}.
  
\begin{thm}[{Ledrappier \cite[page 105]{Ledrappier}}]\label{thm:Ledrappier} Let $V$ be a finite dimensional vector space.
For each H{\"o}lder cocycle $c:\Gamma \times \partial\Gamma \to V$, there exists a H{\"o}lder continuous map $F_c: T^1X_0 \to V$, such that for every $\gamma \in \Gamma -\{e\}$
\[\ell_c(\gamma)=\int_{[\gamma]}F_c.\]
This map $c\to F_c$ induces a bijection between the set of cohomology classes of $V$-valued H{\"o}lder cocycles and the set of cohomology classes of H{\"o}lder maps from $T^1X_{0}$ to $V$.
\end{thm}  
  
For a Hitchin representation $\rho$, this tells us that we can find a H{\"o}lder continuous map $g_{\rho}:T^{1}X_0\to \frak a$ with periods equal to the periods of the vector valued Busemann cocycle $B_{\rho}$. 
Since $h_{\phi\circ g_{\rho}}=h_{B^{\phi}_{\rho}}=h^{\phi}(\rho)$ is finite and positive, by Lemma \ref{lem:PosEntropy}, the reparametrization $\phi\circ g_{\rho}$ on $T^1X_0$ is cohomologous to a positive H{\"o}lder continuous function.

Now we are ready to state our lemma about the existence of positive reparametrization functions, which should be compared to Lemma \ref{lem: rep}.

\begin{lem} \label{lem: rep2} Let $\rho_1$and $\rho_2$ be two different representations in $\cal H_{d}(S)$. There exists a positive H{\"o}lder continuous function $f_{\rho_1}^{\rho_2}: M^\phi_{\rho_1} \to \bb R_{>0}$ such that for every periodic orbit $\tau$ corresponding to $[\gamma]\in [\Gamma]$
\[
\lambda\left(f_{\rho_1}^{\rho_2},\tau\right)=\ell_{\rho_2}^{\phi}([\gamma]).
\]
\end{lem}
\begin{proof}  We denote the translation flow for $M^\phi_{\rho_i}$ with respect to the Busemann cocycle $B_{\rho_i}^{\phi}$ as $\varphi^{i}_t$ for $i=1,2$. Now since $B^{\phi}_{\rho_i}$ has positive entropy, by Theorem \ref{thm:SambarinoReparametrization}, both $\varphi^{i}_t$ are H{\"o}lder conjugate to H{\"o}lder reparametrizations of the geodesic flow on $T^{1}X_0$, where $X_0$ is an auxiliary hyperbolic surface. The reparametrization functions for $\varphi^{i}_t$ on $T^1 X_0$ can be chosen to be positive by the discussion after Theorem \ref{thm:Ledrappier}. Generalizing Remark \ref{rem:ReserveRepa}, we conclude that the flow $\varphi^{2}_t$ is H{\"o}lder conjugate to a H{\"o}lder reparametrization of $\varphi^{1}_{t}$ on $M^\phi_{\rho_1}$. Therefore, there exists a H{\"o}lder function $f\colon M^\phi_{\rho_1}\to \mathbb{R}$ such that the flow $\varphi^{2}_t$ is H{\"o}lder conjugate to $\big(\varphi^{1}_t\big)^f$ where the flow $\big(\varphi^{1}_t\big)^f$ is a reparametrization of $\varphi^{1}_t$ by $f$. Because $h_f=h_{B^{\phi}_{\rho}}$ is positive and finite, we can always choose $f=f_{\rho_1}^{\rho_2}$ in its cohomology class to be a positive function by Lemma \ref{lem:PosEntropy}. One easily checks that $\lambda\left(f,\tau\right)=\ell_{\rho_2}^{\phi}([\gamma])$.
\end{proof}

Once we have defined reparametrization functions, we see that all definitions, remarks and results in Section \ref{ssec:thermo} regarding thermodynamic formalism readily generalize to this context by simply changing the domain from $T^1X_{\rho}$ to $M_{\rho}^\phi$.\

\subsection{Correlation and Manhattan curve theorem revisited}

First, we state the independence lemma for Hitchin representations which generalizes Lemma \ref{thm:indep}.

\begin{lem}[Independence lemma]\label{prop:independenceGeneral} Consider $\phi\in\Delta$ and Hitchin representations $\rho_1,\rho_2\in\cal H_d(S)$ such that $\rho_2\neq \rho_1$ or $\rho_1^\ast$. If there exist $a_1,a_2\in\bb R$ such that $a_1\ell^\phi_{\rho_1}([\gamma])+a_2\ell^\phi_{\rho_2}([\gamma])\in\bb Z$ for all $[\gamma]\in[\Gamma]$, then $a_1=a_2=0$.
\end{lem}

\begin{proof} The Zariski closure $\rm G_i$ of $\rho_i(\Gamma)$ is simple and connected by a result of Guichard (see \cite[Theorem 11.7]{PressureMetric-MainPaper}). Then, we argue by contradiction as in the proof of Lemma \ref{thm:indep}.
\end{proof}

Recall that a flow is weakly mixing if its periods do not generate a discrete subgroup of $\mathbb{R}$. In particular, the independence lemma implies that the translation flow $\varphi_t$ is weakly mixing.

We are now ready to state our correlation theorem for Hitchin representations. Recall that we denote the \emph{renormalized $\phi$ length spectrum} of $\rho$ by $L^{\phi}_\rho=h^{\phi}(\rho)\ell^\phi_\rho$,

\medskip\noindent
{\textbf{Theorem} \ref{thm:HitchinMain}\textbf{.}} 
{\em Given a linear functional $\phi\in\Delta$ and a fixed precision $\eee>0$, for any two different Hitchin representations $\rho_1, \rho_2\colon\Gamma \to \rm{PSL}(d,\mathbb{R})$  such that $\rho_2\neq \rho_1^*$, there exist constants $C=C(\eee, \rho_1,\rho_2,\phi)>0$ and $M=M(\rho_1,\rho_2, \phi) \in (0,1)$ such that }
\[
\# \Big\{[\gamma]\in[\Gamma]\, \Big\vert\, L_{\rho_1}^{\phi}([\gamma]) \in \big(x,x+h^{\phi}(\rho_1)\eee\big),\  L_{\rho_2}^{\phi}([\gamma]) \in \big(x, x+h^{\phi}(\rho_2)\eee\big)\Big\} \sim C \frac{e^{Mx}}{x^{3/2}}.
\]
\begin{proof} 
By Proposition \ref{translationMetricAnosov} and Lemma \ref{prop:independenceGeneral}, the translation flow associated to the linear functional $\phi$ is a weakly mixing metric Anosov flow on $M^\phi_{\rho_1}$ that admits a symbolic coding with H\"older continuous roof function. Moreover, $\rho_1$ and $\rho_2$ are independent thanks to the independence lemma \ref{prop:independenceGeneral}. Thus we can apply Lalley and Sharp's Theorem \ref{prop:lalley} which hold in the context of flows with a symbolic coding with H\"older roof function.

Recall $J(f_{\rho_1}^{\rho_2})$ is the open interval of values $P'(tf_{\rho_1}^{\rho_2})$ for $t\in\mathbb R$. We then want to verify $\frac{h^{\phi}(\rho_1)}{h^{\phi}(\rho_{2})}\in J(f_{\rho_1}^{\rho_2})$. The proof of this fact follows from the same argument as in the proof for Theorem \ref{thm:main}.
\end{proof}
\begin{rem}
Fix $\phi\in\Delta$. Sambarino's orbit counting theorem  \cite[Thm 7.8]{Sambarino_Quantitative} implies that the correlation number $M=M(\rho_1,\rho_2, \phi)$ converges to one and $C(\eee, \rho_1,\rho_2,\phi)$ diverges if $\rho_1$ and $\rho_2$ converges to a Hitchin representation with the same $\phi$-length spectrum.
\end{rem}

For two Hitchin representations $\rho_1$, $\rho_2$ with $\rho_2\neq \rho_1,\rho_1^*$ and $\phi\in\Delta$ consider the {\em Manhattan curve}
\[
\cal C^\phi(\rho_1,\rho_2)=\{(a,b)\in\bb R^2\colon P(-a-bf)=0\}
\]
where $f$ is the reparametrization function from Lemma \ref{lem: rep2}. We obtain a characterization of the correlation number analogous to Theorem \ref{thm:Manhattan}.

\begin{thm} Fix $\phi\in\Delta$, and let $\rho_1$ and $\rho_2$ be Hitchin representations in $\cal H_d(S)$ such that $\rho_2 \neq \rho_1, \rho_1^{*}$. Their correlation number can be written as
\[
M(\rho_1,\rho_2, \phi)=\frac{a}{h^\phi(\rho_1)}+\frac{b}{h^\phi(\rho_2)}
\]
where $(a,b)\in \cal C^\phi(\rho_1,\rho_2)$ is the point on the Manhattan curve at which the tangent line is parallel to the line passing through $(h^\phi(\rho_1),0)$ and $(0,h^\phi(\rho_2))$.
\end{thm}

We conclude by raising two questions which are motivated by Theorem \ref{thm:pinching} and Theorem \ref{thm:cubicrays}.

Recall that Potrie and Sambarino \cite[Thm B]{PotrieSambarino} showed that for every Hitchin representation $\rho$, the simple root lengths are such that $h^{\alpha_i}(\rho)=1$ for $i=1,\dots,d-1$. Moreover, the simple root lengths restrict to the hyperbolic length on the Fuchsian locus $\cal T(S)\subset \cal H_d(S)$. Thus, Theorem \ref{thm:HitchinMain} in this case is a particularly natural generalization of the correlation theorem for hyperbolic surfaces \cite{CorrelationHyperbolic}. Theorem \ref{thm:pinching} exhibits examples of sequences for which the $\alpha_i$-correlation number decays. We ask whether there exist similar examples which lie outside the Fuchsian locus.
\begin{que} \label{Que:SimpleRootLengths}
For $d\geq 3$ and $i=1,\cdots,d-1$, do there exist sequences $(\rho_n)_{n=1}^\infty$ and $(\eta_n)_{n=1}^\infty$ in $\cal H_d(S)$  which leave every compact neighborhood of the Fuchsian locus and such that the correlation numbers satisfy $\lim\limits_{n\to\infty} M(\rho_n,\eta_n,\alpha_i)=0$?
\end{que}

Finally, we raise a conjecture motivated by Theorem \ref{thm:cubicrays}. Similarly to cubic rays introduced in section \ref{sec:applications}, in general for the Hitchin component $\cal H_d(S)$, one can consider a $d$-th order holomorphic differential $q$ over $X_0$ and its associated family of representations $(\rho_t)_{\scriptscriptstyle t\in\bb R}$ in $\cal H_d(S)$ parametrized by $(tq)_{t\in\mathbb{R}}$. These representations are given by holonomies of {\em cyclic Higgs bundles}. We refer to \cite{Baraglia_cyclic} for a detailed definition and discussion. When $t\geq 0$, we say this family of representations $(\rho_t)_{t\geq 0}$ is a \emph{ray associated to $q$}.  

\begin{conj}
\label{Conj:cyclic}  Let $(\rho_t)_{t\geq 0}$ and $(\eta_t)_{t\geq 0}$ be two rays associated to two different $d$-th order holomorphic differentials $q_1$ and $q_2$ on a hyperbolic structure $X_0$ such that $q_1,q_2$ have unit $L^2$-norm with respect to $X_0$ and $q_1\neq -q_2$. Then the correlation number $M(\rho_{t},\eta_{t}, \sum_{i=1}^{d-1}\alpha_i)$ is uniformly bounded away from zero as $t$ goes to infinity.
\end{conj}

\subsection*{Acknowledgements} We would like to thank Harrison Bray, Richard Canary, Le{\'o}n Carvajales and Michael Wolf for several insightful conversations and suggestions. We are very grateful to Richard Canary and an anonymous referee for their comments on early versions of this manuscript.
GM acknowledges partial support by the American Mathematical Society and the Simons Foundation. XD wants to thank Rice math department for providing support during the preparation of this paper.

\bibliography{mybib}
\bibliographystyle{abbrv}
\end{document}